\documentclass[11pt]{article}
\usepackage{amsfonts}
\usepackage{graphics}
\usepackage{indentfirst}
\usepackage{cite}
\usepackage{latexsym}
\usepackage{amsmath}
\usepackage{amssymb}
\usepackage[dvips]{epsfig}
\usepackage{amscd}
\usepackage[marginal]{footmisc}
\usepackage[bookmarks,colorlinks,citecolor=blue,linkcolor=red]{hyperref}
\usepackage{mathtools}
\allowdisplaybreaks
\newtheorem{theorem}{Theorem}[section]
\newtheorem{remark}{Remark}[section]
\newtheorem{definition}{Definition}[section]
\newtheorem{lemma}[theorem]{Lemma}

\newtheorem{proposition}[theorem]{Proposition}

\def\O{{\Omega }}

\def\norm[#1]#2{\|#2\|_{#1}}

\def\la{\label}

\def\na{\nabla}

\def\QEDopen{{\setlength{\fboxsep}{0pt}\setlength{\fboxrule}{0.2pt}\fbox{\rule[0pt]{0pt}{1.3ex}\rule[0pt]{1.3ex}{0pt}}}} 
\def\QED{\QEDopen} 
\def\endproof{\hspace*{\fill}~\QED\par\endtrivlist\unskip}

\newcommand{\lm}{\lambda}
\newcommand{\ltwo}{_{L^2}^2}
\newcommand{\pa}{\partial}
\renewcommand{\r}{\mathbb{R}}

\newcommand{\ga}{\gamma}
\newcommand{\curl}{{\rm curl} }

\newcommand{\si}{\sigma}

\newcommand{\ol}{\overline}
\newcommand{\bt}{\begin{theorem}}
\newcommand{\bl}{\begin{lemma}}
\newcommand{\el}{\end{lemma}}
\newcommand{\et}{\end{theorem}}
\newcommand{\bn}{\begin{eqnarray}}
\newcommand{\en}{\end{eqnarray}}
\newcommand{\bnn}{\begin{eqnarray*}}
\newcommand{\enn}{\end{eqnarray*}}
\newcommand{\bnnn}{\begin{eqnarray*}}
\newcommand{\ennn}{\end{eqnarray*}}
\newcommand{\ba}{\begin{aligned}}
\newcommand{\ea}{\end{aligned}}
\newcommand{\be}{\begin{equation}}
\newcommand{\ee}{\end{equation}}

\renewcommand{\thefootnote}{}
\newcommand\blfootnote[1]{%
  \begingroup
  \renewcommand\thefootnote{}\footnote{#1}%
  \addtocounter{footnote}{-1}%
  \endgroup
}
\makeatletter      
\@addtoreset{equation}{section}
\makeatother       

\title{Global Strong and Weak Solutions to the Initial-boundary-value Problem of 2D Compressible MHD System with Large Initial Data and Vacuum}

\author{Yazhou C{\small HEN}, Bin H{\small UANG}, Xiaoding S{\small HI}   \\[3mm] {\normalsize   College of Mathematics and Physics, }\\ {\normalsize  Beijing University of Chemical Technology, Beijing 100029, P. R. China} }

\date{ }

\begin{document}

\maketitle

\blfootnote{Email: chenyz@mail.buct.edu.cn (Y.Chen), abinhuang@gmail.com (B.Huang), shixd@mail.buct.edu.cn (X.Shi)}

\begin{abstract}
In this paper, we study the barotropic compressible magnetohydrodynamic equations with the shear viscosity being a positive constant and the bulk one being proportional to a power of the density in a general two-dimensional bounded simply connected domain. For initial density allowed to vanish, we prove that the initial-boundary-value problem of 2D compressible MHD system admits the global strong and weak solutions without any restrictions on the size of initial data provided the shear viscosity is a positive constant and bulk one is $\lambda=\rho^\beta$ with $\beta>4/3$. As we known, this is the first result concerning the global existence of strong solutions to the compressible MHD system in general two-dimensional bounded domains with large initial data and vacuum.
\end{abstract}



\section{Introduction}
We deal with the viscous barotropic compressible magnetohydrodynamic (MHD) equations for isentropic flows in a domain $\Omega\subset\r^{2}$, which can be written as
\begin{equation}\label{CMHD}
\begin{cases}
\rho_t+ \mathop{\mathrm{div}}\nolimits(\rho u)=0,\\
(\rho u)_t+\mathop{\mathrm{div}}\nolimits(\rho u\otimes u)+\nabla P
=\mu \Delta u+\nabla((\mu+\lambda) \mathop{\mathrm{div}}\nolimits u)+H \cdot \nabla H - \frac{1}{2}\nabla|H|^2,\\
H_t+u \cdot\nabla H-H \cdot \nabla u+ H \mathop{\mathrm{div}}\nolimits u=\nu \Delta H,
\\
\mathop{\mathrm{div}}\nolimits H=0,
\end{cases}
\end{equation}
where $(x,t)\in\Omega\times (0,T)$, $t\geq 0$ is time, and $x=(x_1,x_2)$ is the spatial coordinate. The unknown functions $\rho=\rho(x,t), u=(u_1,u_2)(x,t),$ and $H=(H_1,H_2)(x,t)$ denote the fluid density, velocity and magnetic field respectively and the pressure $P$ is given by
\begin{equation}
P=P(\rho)=a\rho^{\gamma}, 	
 \end{equation}
with the constants $a>0$ and $\gamma >1$.
The shear viscosity $\mu$ and the bulk one $\lambda$ satisfy the following hypothesis (see \cite{vk1995}):
\begin{equation}\label{r-b-0}
\mu=const.>0,\quad \lambda=\lambda(\rho)=b\rho^\beta,
\end{equation}
with the positive constants $b$ and $\gamma$.
The constant $\nu >0$ is the resistivity coefficient which is inversely proportional to the electrical conductivity constant. In what follows, without loss of generality, we set $a=b=1$.
Throughout this paper,  $\Omega\subset\r^{2}$ denotes a simply connected bounded domain with $C^{2,1}$ boundary $\partial \Omega$.
In addition, this paper concerns the problem of \eqref{CMHD}-\eqref{r-b-0} with the initial data
\begin{equation}\label{initial}
\displaystyle  (\rho,\rho u, H)\big|_{t=0}=(\rho_0, m_0,H_0),\quad \text{in}\,\,\, \Omega,
\end{equation}
and the boundary conditions
\begin{align}
 &u\cdot n=0,\,\,\,\curl u=0, &\text{on} \,\,\,\partial\Omega\times(0,T), \label{navier-b}\\
&\quad H=0 
, &\text{on} \,\,\,\partial\Omega\times(0,T),\label{boundary}
\end{align}
where $n=(n_1,n_2)$ is the unit outward normal vector to $\partial \Omega$.

The compressible MHD system \eqref{CMHD} plays a fundamental role in astrophysics, geophysics and plasma physics and its mathematical challenges has attracted a lot of attention of mathematicians. There are a growing literature devoted to the analysis of the well-posedness and dynamic behavior to the solutions of the system, see, for example, \cite{ht2005,ko1982,cw2002,cw2003,fhl2017,djj2013,df2006,fy2008,fy2009,hhpz2017,hw2008,hw2008-1,hw2010,k1984,lyz2013,sh2012,lxz2013,liu2015,lh2015,lsx2016,tg2016,vh1972,wang2003,xh2017,zjx2009} and their references.
Among them, we briefly review some results related to well-posedness of solutions for the multi-dimensional compressible MHD equation with constant viscosities.
For Cauchy problem, Vol'pert-Hudjaev \cite{vh1972} and Fan-Yu \cite{fy2009} obtained the local existence of  classical solutions to the 3D compressible MHD equations with the initial density is strictly positive or could contain vacuum, respectively.
Kawashima \cite{k1984} first established the global smooth solutions to the general electro-magneto-fluid equations in two dimensions with non-vacuum.
Suen-Hoff\cite{sh2012} and Liu et al.\cite{lyz2013} obtained the weak solutions to the 3D compressible magnetohydrodynamic flows with discontinuous initial data.
Recently, Li et al.\cite{lxz2013} and Lv et al.\cite{lsx2016} established the global existence of classical solutions with large oscillations and vacuum for 3D case and 2D one, respectively, provided the initial data be of small energy.
For the initial-boundary-value problem with non-slip boundary condition for the velocity,
Hu-Wang \cite{hw2010} proved the global existence of renormalized solutions for general large initial data, also see \cite{hw2008,fy2008} for the non-isentropic compressible MHD equations.
As far as the slip boundary is concerned, 
Tang-Gao \cite{tg2016} obtained the local strong solutions to the compressible MHD equations in a 3D bounded domain with the Navier-slip condition.
Dou et al.\cite{djj2013} prove the global existence and uniqueness of smooth solutions around a rest state in a 2D bounded domain with slip boundary condition.
More recently, Chen et al.\cite{chs2020} obtained the global classical solutions with vacuum and small energy but possibly large oscillations in a 3D bounded domain with slip boundary condition, which generalized the results of Cai-Li \cite{cl2019} for the barotropic compressible Navier-Stokes equations to the compressible MHD
ones.

Until now, all the global existence of strong (classical) solutions for the  multi-dimensional compressible MHD equations were obtained with some ``smallness" assumptions on the initial data. In contrast, positive results without limitation on the size of initial value are rather fewer.
It should be noted that when $H=0$, the compressible MHD system \eqref{CMHD} turns to be the compressible Navier-Stokes equations with density-dependent viscosity.
Vaigant-Kazhikov \cite{vk1995} first obtained a unique global strong solution away from vacuum for 2D compressible Navier-Stokes equation in rectangle domain with no restrictions on the size of initial data provided $\beta>3$.
Recently, Huang-Li \cite{hl2016,hl2012} applied some new ideas based on commutator theory and blow up criterion, and improved the conclusion in periodic case, and even for the Cauchy problem in the whole space (also see \cite{jwx2018}), demanding only $\beta>\frac{4}{3}$. Up to now, $\beta>\frac{4}{3}$ still seems to be the best result one may expect. Let us turn back to the compressible MHD system with density-dependent viscosity. Lv-Huang \cite{lh2015} obtained the local strong solutions of Cauchy problem of the two-dimensional compressible MHD equations with vacuum as far field density. Mei\cite{mei2015} established the global well-posedness of classical solutions to the 2D compressible MHD equations with large initial data and vacuum on the torus $\mathbb{T}^2$ and the whole space $\r^2$.

However, for general domains, the theory of large initial data for compressible MHD system is still blank, due to boundary terms do bring some essential difficulties and the classical commutator theory in the case of the whole space $\r^2$ and the torus $\mathbb{T}^{2}$ is no longer available for general bounded domains.
Very recently, Fan et al.\cite{fll2021} study the global existence of strong and weak solutions of compressible Navier-Stokes system with large initial data in general simply connected domains. They get a pointwise representation of the effective viscous flux via applying Green's function, as a substitute for the commutator for the case of $\mathbb{T}^{2}$ and $\mathbb{R}^{2}$, which plays an important role to derive the upper bound of density $\rho$. In particular, for the unit disc $\mathbb{D}$, Green's function takes the following form (see \cite{stt2016}):
\begin{equation}\label{gre1}
G(x, y)=-\frac{1}{2\pi}\Big(\ln|x-y|+\ln\big||x|y-\frac{x}{|x|}\big|\Big).
\end{equation}
For the Neumann problem in $\mathbb{D}$ as follows:
\begin{equation}\label{flux1}
\begin{cases}
\triangle F=\mathrm{div}f &\mathrm{in}\, \,  \mathbb{D}, \\
\frac{\partial F}{\partial n}=f\cdot n &\mathrm{on}\, \,  \partial \mathbb{D},
\end{cases}
\end{equation}
the Green's identity yields that $F$ has the following integral representation:
\begin{equation}\label{flux2}
F(x)=  -\int_\mathbb{D} \nabla_y G(x, y)\cdot f(y)dy+\int_{\partial\mathbb{D}}\frac{\partial G}{\partial n}(x, y)F(y)dS_y.
\end{equation}
For the general simply connected domain $\Omega$, the similar integral representation has been derived by applying the Riemann mapping theorem and  the pull-back Green's function method (see \cite[Lemma 3.7]{fll2021}).
Fortunately, these methods are still available to get a pointwise representation of the effective viscous flux for compressible MHD system, defined by
\begin{equation}\label{flux}
\displaystyle  F\triangleq(2\mu+\lambda)\mathrm{div} u -P-\frac{1}{2}|H|^2.
\end{equation}

Motivated by the interesting work of Fan et al.\cite{fll2021} for global existence to compressible Navier-Stokes system with large data and vacuum, the main purpose of this paper is to establish the global well-posedness of strong and weak solutions of the compressible MHD system \eqref{CMHD}-\eqref{boundary} in a simply connected bounded domain $\Omega\subset\r^2$ without any limitation on the size of initial value. The initial density is allowed to contain vacuum states.

Before formulating our main result, we first explain the notation and conventions used throughout the paper.
For integer $k\geq 1$ and $1\leq q<+\infty$, We denote the standard Sobolev space by $W^{k,q}(\Omega)$ and $H^k(\Omega)\triangleq W^{k,2}(\Omega)$.
For simplicity, we denote $L^q(\Omega)$, $W^{k,q}(\Omega)$ and $H^k(\Omega)$ by $L^q$, $W^{k,q}$ and $H^k$ respectively, and set
$$\int_\Omega fdx \triangleq \int fdx,\quad \bar{f}\triangleq \frac{1}{|\Omega|}\int fdx 
. $$
For $v=(v_1,v_2)$, the material derivative and the transpose of $v$, the transpose gradient and the vorticity are given by
\begin{equation}\label{bot}
\displaystyle \frac{D}{Dt}v=\dot{v}\triangleq v_t+u\cdot \nabla v,
\ v^{\bot}=(-v_2,v_1),\
\nabla^{\bot} \triangleq(-\partial_2,\partial_1),\ \omega\triangleq\nabla^{\bot}\cdot v.
\end{equation}

Finally, we give the definition of weak and strong solution to \eqref{CMHD} as follows.
\begin{definition}
$(\rho, u, H)$ is called a weak solution to \eqref{CMHD} if it satisfies \eqref{CMHD} in the sense of distribution. Moreover, if all the derivatives of the weak solution involved in \eqref{CMHD} are regular distributions, and \eqref{CMHD} hold almost everywhere in $\Omega\times (0,  T)$, then the solution is called strong.
\end{definition}

Now we can state our main result, Theorem \ref{th1}, concerning existence of global strong solutions to the problem  \eqref{CMHD}-\eqref{boundary}.
\begin{theorem}\label{th1}
Let $\Omega$ be a simply connected bounded domain in $\r^2$ with $C^{2,1}$ boundary $\partial\Omega$. Assume that
\begin{equation}\label{b-g}
 \displaystyle \beta>\frac{4}{3},\quad \gamma>1,
 \end{equation}
and the initial data $(\rho_0,m_0,H_0)$ satisfy
\begin{align}\label{dt1}
0\leq\rho_0\in W^{1,q},\quad m_0=\rho_0 u_0,\quad (u_0, H_0)\in H^1,\quad \mathrm{div} H_0=0,
\end{align}
and the boundary conditions \eqref{navier-b}-\eqref{boundary}.
Then the initial-boundary-value problem \eqref{CMHD}-\eqref{boundary} has a unique global strong solution $(\rho,u,H)$ in $\Omega\times(0,\infty)$ satisfying for any $0<T<\infty$ and $q>2$,
\begin{equation}\label{esti-uh}
\begin{cases}
\rho\in C([0, T]; W^{1, q} ), \, \rho_t\in L^\infty( 0, T ;L^2 ), \\
(u,H)\in L^\infty(0, T;H^1 )\cap L^{1+1/q}(0, T;W^{2, q} ), \\
(\sqrt{t}u,\sqrt{t}H)\in L^\infty(0, T;H^{2} )\cap L^2(0, T;W^{2, q} ), \\
(\sqrt{t}u_t,\sqrt{t}H_t)\in L^2(0, T;H^1 ), \\
\rho u\in C([0,T];L^2),\quad H\in C([0,T];H^1),\\
(\sqrt{\rho}u_t,H_t)\in L^2(\Omega\times(0, T)).
\end{cases}
\end{equation}
\end{theorem}

The second result concerns the global existence of weak solution to the problem  \eqref{CMHD}-\eqref{boundary}.
\begin{theorem}\label{th2}
Assume that \eqref{b-g} holds and  the initial data $(\rho_0,m_0,H_0)$ satisfy
\begin{align}\label{dt1-2}
0\leq\rho_0\in L^\infty,\quad m_0=\rho_0 u_0,\quad (u_0, H_0)\in H^1,\quad \mathrm{div} H_0=0,
\end{align}
and the boundary conditions \eqref{navier-b}-\eqref{boundary}.
Then the problem \eqref{CMHD}-\eqref{boundary} has at least one weak solution $(\rho,u,H)$ in $\Omega\times(0,\infty)$ satisfying for any $0<\tau <T<\infty$ and $p\geq 1$,
\begin{equation}\label{esti-w}
\begin{cases}
\rho\in L^\infty(0,T; L^\infty)\cap C([0,T];L^{p} ),\\
(u,H)\in L^\infty(0, T;H^1 ), \\
(u_t,H_t)\in L^2(\tau, T;L^2 ), \\
(\nabla u,\nabla H)\in L^\infty(\tau, T;L^p),\\
(F,\omega)\in L^2(0, T;H^1).
\end{cases}
\end{equation}
\end{theorem}

\begin{remark}\label{rem:1}
Compared with \cite{mei2015} on the torus $\mathbb{T}^2$ and the whole space $\r^2$, our Theorem  \ref{th1} seems to be the first concerning the global existence of strong solutions to the compressible MHD system  in general two-dimensional bounded domains with large data and vacuum.
\end{remark}

\begin{remark}\label{rem:2}
When $H = 0$, i.e., there is no electromagnetic field effect, the compressible MHD system \eqref{CMHD} turns to be the compressible Navier-Stokes equations, and Theorem \ref{th1} is the same as the result of Fan et al.\cite{fll2021}. Roughly speaking, we generalize the results of \cite{fll2021} to the compressible MHD equations.
\end{remark}

\begin{remark}
Similarly to that for compressible Navier-Stokes equations (cf.\cite{hl2016,fll2021}), we only require the initial data $(\rho_0,m_0,H_0)$ satisfy the natural compatibility condition $m_0=\rho_0u_0$, which is much weaker than those of \cite{lxz2013,chs2020,mei2015} where
\begin{align}\label{dt3}
\displaystyle  -\mu\triangle u_0-\nabla ((\mu+\lambda(\rho_0))\mathop{\mathrm{div}}\nolimits u_0 )+ \nabla P(\rho_0)- H_0 \cdot \nabla H_0+\frac{1}{2}\nabla|H_0|^2  = \rho_0^{1/2}g,
\end{align}
for some $g\in L^2$ is required.
\end{remark}

\begin{remark}
Similar to \cite{chs2020,mei2015}, if the initial data $(\rho_0,m_0,H_0)$ satisfies some additional regularity and compatibility conditions \eqref{dt3}, the global strong solutions obtained by Theorem \ref{th1} become classical ones for positive time.
\end{remark}

\begin{remark}
For technical reason, we still need assume that $\beta>4/3$ which is the same as that of \cite{hl2016,jwx2018,fll2021}. However, it seems that $\beta>1$ is the extremal case for the system \eqref{CMHD} (see \cite{vk1995} or Lemma \ref{lem-rhop} below). Therefore, it would be interesting to study the case of $1<\beta\le 4/3$ which is left for the future.
\end{remark}

We now sketch the main idea used in this paper. Similar to the argument in \cite{hl2016,fll2021,mei2015}, the key issue in our proof is to derive the upper bound of the density in Proposition \ref{lem-rho-inf}. As mentioned above, unlike \cite{hl2016,mei2015}, it can not apply the standard commutator theory to estimate the density. Motivated by the work on the compressible Navier-Stokes equations in \cite{fll2021}, we get a pointwise representation of the effective viscous flux for compressible MHD system (see \eqref{f}). However, compared with the compressible Navier-Stokes equations, some additional difficulties will arise when we deal with the strong coupling and interplay interaction between the fluid motion and the magnetic field.
The following key observations help us to deal with the boundary terms and the interaction of the magnet field and the velocity field very well.
First, we obtain the estimate on $L{^\infty}(0,T;L{^p})$-norm $(p\geq 2)$ of the magnetic field $H$ in 2D bounded domain with Dirichlet condition (see Lemma \eqref{lem-h-lp}).
It is essential to ensure that we follows the idea \cite{vk1995,jwx2014,mei2015} to get the estimate on the $L^\infty(0,T;L^p(\Omega))$-norm $(p\geq 1)$ of the density $\rho$ (see Lemma \eqref{lem-rhop}). It also can used to deal with the magnetic force and the convection term, such as the $H \cdot \nabla H$ and $u\cdot \nabla H$.
Moreover, in view of the structure of the magnet equation $\eqref{CMHD}_3$, we can obtain the estimate on the $L{^2}(0,T;L{^2})$-norm of $\Delta H$ and $H_t$ and observe that $L{^2}$-norm of $\nabla^2 H$ is equivalent to that of $\Delta H$.
Second, thanks to \cite{Aramaki2014,Von1992}, Lemma \ref{lem-vn} allows us to control $\nabla u$ by means of $\mathrm{div} u$ and $\curl u$ due to the boundary condition $u\cdot n=0$ on $\partial\Omega$. Furthermore, this boundary condition helps us reduce the integral representation to the desired commutator form. Thus we may derive the precise control of $F$. Third, we deduce from the momentum equation $\eqref{CMHD}_2$ and the slip boundary \eqref{navier-b} in 2D case that the effective  viscous flux $F$ and the vorticity $\omega$ solves the Neumann problem and the Dirichlet problem respectively (see \eqref{de-f} and \eqref{de-om}). Then standard $L^p$ theory yields the estimates of $\nabla F$ and $\nabla\omega$.
In addition, since $u\cdot n=0$ on $\partial\Omega$, it follows that
\begin{align}\label{bdd2}
\displaystyle u\cdot\nabla u\cdot n=-u\cdot\nabla n\cdot u,\quad u=(u \cdot n^{\bot})n^{\bot}.
\end{align}
As observed in \cite{cl2019}, this is the key to estimate the integrals on the boundary $\partial\Omega$, especially the trace of $\nabla u$ on $\partial\Omega$ (see \eqref{a1-i1},\eqref{J11-1}, \eqref{J11-2}).
Similarly, one can get
\begin{align}\label{bdd3}
\displaystyle \big(\dot{u}-(u\cdot n^{\bot})u \cdot \nabla n^{\bot}\big)\cdot n=0,
\end{align}
which yields the estimate of $\dot{u}$ and $\nabla \dot{u}$ (see \eqref{udot} and \eqref{tdudot}).
Finally, in order to estimate the derivatives of the solutions, we recall the similar Beale-Kato-Majda-type inequality with the respect to the slip boundary condition to prove the important estimates on the gradients of the density (see \eqref{bkm}).

An outline of the paper is as follows.
In Section \ref{se2}, we list some elementary inequalities and important lemmas that we use intensively in the paper.
Section \ref{se3} is devoted to deriving the upper bounded of the density which plays an essential role in the whole procedure.
Based on the previous estimates, the lower and higher order estimates are established in Section \ref{se4}.
Finally, the proof of Theorem \ref{th1}-\ref{th2} will be completed in Section \ref{se5}.

\section{Preliminaries}\label{se2}
In this section, we review some elementary inequalities and important lemmas that are used extensively in this paper.
To begin with, the following local existence theory of \eqref{CMHD}, where the initial density is strictly away from vacuum, can be found in \cite{fy2009,k1984,djj2013}.
\begin{lemma}\label{lem-local}
Assume that the initial data $(\rho_0,u_0,H_0)$ satisfies
\begin{align}\label{dt2}
\inf\limits_{x\in \Omega}\rho_0(x)>0,\quad\rho_0\in H^2,\quad m_0=\rho_0 u_0,\quad (u_0, H_0)\in H^2,\quad \mathrm{div} H_0=0,
\end{align}
and the boundary conditions \eqref{navier-b}-\eqref{boundary}.
Then there exist a small time $T_0>0$ and a constant $C_0>0$ both depending only on $\Omega$, $T$, $\mu$, $\beta$, $\gamma$, $\nu$, $\|\rho_0\|_{H^2}$, $\|u_0\|_{H^2}$, $\|H_0\|_{H^2}$, and $\inf\limits_{x\in \Omega}\rho_0(x)$ such that there exists a unique strong solution $(\rho,u,H)$ of the system \eqref{CMHD}-\eqref{boundary} in $\Omega \times (0,T_0]$, satisfying that
\begin{equation}\label{esti-uh-local}
\begin{cases}
\rho\in C([0, T_0]; H^2 ), \, \rho_t\in C([0, T_0]; H^1 ), \\
(u,H)\in L^\infty(0, T_0;H^2 )\cap L^{2}(0, T_0;H^3), \\
(u_t,H_t)\in L^2(0, T_0;H^2 ), \\
(u_{tt},H_{tt})\in L^2((0, T_0)\times\Omega),
\end{cases}
\end{equation}
and
\begin{equation}\label{rho-local}
	\inf_{(x,t)\in \Omega\times (0,T_0)} \rho(x,t)\geq C_0>0.
\end{equation}
\end{lemma}

Next, the following well-known Gagliardo-Nirenberg's inequality (see \cite{nir1959,tale1976}) will be used later.
\begin{lemma}\label{lem-gn}
Assume that $\Omega$ is a bounded Lipschitz domain in $\r^2$. For  $p\in [2,\infty),q\in(1,\infty), $ and
$ r\in  (2,\infty),$ there exists some generic
 constant
$C>0$ which may depend  on $p,\,\,q,$ and $r$ such that for   $f\in H^1({\O }) $
and $g\in  L^q(\O )\cap D^{1,r}(\O), $    we have
\be\la{g1}\|f\|_{L^p(\O)}\le C_1 p^{\frac{1}{2}}\|f\|_{L^2}^{\frac{2}{p}}\|\na
f\|_{L^2}^{1-\frac{2}{p}}+C_2p^{\frac{1}{2}}\|f\|_{L^2} ,\ee
\be\la{g2}\|g\|_{C\left(\ol{\O }\right)} \le C_1
\|g\|_{L^q}^{q(r-2)/(2r+q(r-2))}\|\na g\|_{L^r}^{2r/(2r+q(r-2))} + C_2\|g\|_{L^2}.
\ee
Moreover, if $f\cdot n|_{\partial\Omega}=0,\,\,\,g\cdot n|_{\partial\Omega}=0,$ or $\bar{f}=0$, $\bar{g}=0$, then the constant $C_2=0.$
\end{lemma}

Next, we need the $L^p$ theory for the div-curl system to control the gradient of velocity. The following conclusion is given in \cite{Von1992,Aramaki2014}.
\begin{lemma}\label{lem-vn}
Let $1<q<+\infty,$ $\Omega$ is a bounded domain in $\r^2$ with Lipschitz boundary $\partial\Omega.$ For $v\in W^{1,q}$, if $\Omega$ is simply connected and $v\cdot n=0$ on $\partial\Omega$, then it holds that
\begin{equation}\label{tdu1}
\|\nabla v\|_{L^q}\leq C(\|\mathrm{div} v\|_{L^q}+\|\curl v\|_{L^q}).	
\end{equation}
\end{lemma}

Besides, we require the Beale-Kato-Majda type inequality with respect to the slip boundary condition \eqref{navier-b} which is given in \cite{BKM1984,cl2019}.
\begin{lemma}\label{lem-bkm}
For $2<q<\infty$, $\Omega$ is a bounded domain in $\r^2$ with Lipschitz boundary $\partial \Omega$, assume that $u\cdot n=0$ and $\curl u=0$ on $\partial\Omega$, $\nabla u\in W^{1,q}(\Omega)$, then there is a constant  $C=C(q)$ such that  the following estimate holds
\begin{equation}
\|\na u\|_{L^\infty}\le C\left(\|{\rm div}u\|_{L^\infty}+\|\curl u\|_{L^\infty} \right)\ln(e+\|\na^2u\|_{L^q})+C\|\na u\|_{L^2} +C .\label{bkm}	
\end{equation}
\end{lemma}

Moreover, in order to estimate the material derivative of $u$, we review the following Poincar\'{e}-type inequality of $\dot u$ (see \cite{cl2019,fll2021}).
\begin{lemma}\label{lem-ud}
For $p\ge 1,$
there exist   positive constants $C_1(p,\Omega)$ and  $C_2(\Omega)$ such that
\begin{align}
&\|\dot{u}\|_{L^p}\le C(\|\nabla\dot{u}\|_{L^2}+\|\nabla u\|_{L^2}^2),\label{udot}\\
&\|\nabla\dot{u}\|_{L^2}\le C(\|\mathrm{div} \dot{u}\|_{L^2}+\|\curl \dot{u}\|_{L^2}+\|\nabla u\|_{L^4}^2).\label{tdudot}
\end{align}
\end{lemma}

To this end, following the idea of \cite{fll2021}, we recall the Riemann mapping theorem (\cite[Chapter 9]{ss2003}) to introduce the conformal mapping $\varphi=(\varphi_1, \varphi_2):\bar{\Omega}\rightarrow\bar{\mathbb{D}}$, which ensures that we can pull back Green's function of the unit disc $\mathbb{D}$ and reduce the general case to that of unit disc. For the sake of completeness, we require the following lemma (\cite{ss2003,wars1961,fll2021}).

\begin{lemma}\label{lem-varphi}
The conformal mapping $\varphi=(\varphi_1, \varphi_2):\bar{\Omega}\rightarrow\bar{\mathbb{D}}$ is a smooth function satisfying the following Cauchy-Riemann equations:
\begin{equation}\label{c-r}
\begin{cases}
\partial_1\varphi_1=\partial_2\varphi_2, \\
\partial_2\varphi_1=-\partial_1\varphi_2.
       \end{cases}
\end{equation}
and shares the following crucial properties:

(i) $\varphi(x)$ is a one to one holomorphic mapping from $\Omega$ to $\mathbb{D}$ and maps the boundary $\partial\Omega$ onto the boundary $\partial\mathbb{D}$.

(ii) For any integer $k=0,1,2$, there exists some constant $C$ depending only on $\Omega$ and $k$ such that
\begin{equation}\label{vp1}
|\nabla^k\varphi(x)|\leq C,\,\,\, \forall x\in\bar{\Omega}.
\end{equation}
Consequently,
\begin{equation}\label{vp2}
|\nabla^k\varphi(x)-\nabla^k\varphi(y)|\leq C|x-y|,\,\,\, \forall x, y\in\bar{\Omega}.
\end{equation}

(iii) There exist  two positive constants $c_1, c_2$ such that
\begin{equation}\label{vp3}
c_1|x-y|\leq|\varphi(x)-\varphi(y)|\leq c_2|x-y|.
\end{equation}

(iv) Angle is preserved by $\varphi(x)$, that is, for any two smooth curves $\gamma_1(t), \gamma_2(t):(0, 1)\rightarrow\bar{\Omega}$,
\begin{equation}\label{angle}
\langle\frac{d}{dt}\gamma_1(t), \frac{d}{dt}\gamma_2(t)\rangle
=\langle\frac{d}{dt}\varphi(\gamma_1(t)), \frac{d}{dt}\varphi(\gamma_2(t))\rangle,\,\,\, \forall t\in(0, 1).
\end{equation}

(v) For any harmonic function $h(y)$ in $\mathbb{D}$, $h(\varphi(x))$ is still harmonic in $\Omega$, that is,
\begin{equation}\label{harm}
\Delta h(y)=0\ \mathrm{in}\ \mathbb{D}\Rightarrow\Delta h(\varphi(x))=0\ \mathrm{in}\ \Omega.
\end{equation}

(vi) Let $n$ and $\widetilde{n}$ be the unit outer normal vector of $\partial\Omega$ at $x_0$ and $\partial\mathbb{D}$ at $\varphi(x_0)$ respectively.  Then
\be  \label{vp4}
n\cdot\nabla\varphi(x_0)=|\nabla\varphi_1(x_0)|\,\widetilde{n}.
\ee
\end{lemma}

\section{A priori estimates (I): upper bound of the density}\label{se3}
In this section, we will establish the upper bound of density which is independent of the lower one of the initial density. In what follows, we always assume that $(\rho_0,u_0,H_0)$ satisfies \eqref{dt2} and $(\rho,u,H)$ is the strong solution to \eqref{CMHD}-\eqref{boundary} on $\Omega\times(0,T]$ obtained by Lemma \eqref{lem-local}.

We review that $F$ and $\omega$ are the effective viscous flux and the vorticity respectively as follows:
\begin{equation}
F\triangleq (2\mu+\lambda)\mathop{\mathrm{div}}\nolimits u-P-\frac{1}{2}|H|^2,
\quad \omega \triangleq\nabla^{\bot}\cdot u=\partial_1 u_2-\partial_2 u_1.
\end{equation}
Thus we define
\begin{align}
&  A_1(t) \triangleq e+\|F/\sqrt{2\mu+\lambda(\rho)}(t)\|_{L^2} +\|\omega(t)\|_{L^2} +\|\nabla H(t)\|_{L^2},\label{As1}\\
& A_2(t)  \triangleq \|\sqrt{\rho}\dot{u}(t)\|_{L^2}+\|\Delta H(t)\|_{L^2}+\|H_t(t)\|_{L^2},\label{As2}\\
& R_T \triangleq 1+ \sup_{  0\le t\le T   } \|\rho(t)\|_{L^\infty}.\label{As3}
\end{align}

Before going further, we rewrite \eqref{CMHD} in the following form:
\begin{equation}\label{CMHD1}
\begin{cases}
P_t+ \mathop{\mathrm{div}}\nolimits(P u)+(\gamma-1)P \mathop{\mathrm{div}}\nolimits u=0,\\
\rho u_t+\rho u \cdot \nabla u - \nabla((2\mu+\lambda(\rho))\mathop{\mathrm{div}}\nolimits u)-\mu\nabla^{\bot}\omega + \nabla P=H \cdot \nabla H -\frac{1}{2}\nabla|H|^2,\\
H_t-\nu \Delta H= H \cdot \nabla u-u \cdot\nabla H-H \mathop{\mathrm{div}}\nolimits u,
\\
\mathop{\mathrm{div}}\nolimits H=0,
\end{cases}
\end{equation}
Multiplying $\eqref{CMHD1}_1$ by $\frac{1}{\gamma-1}$, $\eqref{CMHD1}_2$ by $u$ and $\eqref{CMHD1}_3$ by $H$ respectively, integrating by parts over $\Omega$, summing them up, in view of \eqref{navier-b} and \eqref{boundary}, we obtain that
\begin{equation}\label{m2}
\begin{split}
\displaystyle  &\left(\int \Big(\frac{\rho^\gamma}{\gamma-1}+\frac{1}{2}\rho |u|^{2}+\frac{1}{2}|H|^{2}\Big)dx\right)_t  \\
&+ \int(\lambda(\rho) + 2\mu)(\mathrm{div} u)^{2}dx+ \mu\int\omega^{2}dx + \nu\int|\nabla H|^{2}dx =0,	
\end{split}
\end{equation}
which, integrated over $(0,T)$, leads to the following elementary energy estimates.
\begin{lemma}\label{lem-basic}
 Let $(\rho,u,H)$ be a smooth solution of \eqref{CMHD}-\eqref{boundary} on $\O \times (0,T]$. Then
\begin{equation}\label{basic1}
\begin{split}
\displaystyle  &\sup_{0\le t\le T}
\left(\frac{1}{2}\|\sqrt{\rho}u\|_{L^2}^2+\frac{1}{\gamma-1}\|\rho\|_{L^\gamma}^\gamma+\frac{1}{2}\|H\|_{L^2}^2\right)\\
 &+ \int_0^{T}\|\sqrt{2\mu+\lambda(\rho)}\mathrm{div} u\|_{L^2}^2 +\mu\|\omega\|_{L^2}^2+\nu \|\nabla H\|_{L^2}^2)dt \le E_0,
\end{split}
\end{equation}
where $E_0\triangleq\frac{1}{2}\|\sqrt{\rho_0}u_0\|_{L^2}^2+\frac{1}{\gamma-1}\|\rho_0\|_{L^\gamma}^\gamma+\frac{1}{2}\|H_0\|_{L^2}^2$.
\end{lemma}

\begin{remark}\label{rem:energy}
According to Lemma \ref{lem-vn}, it follows from \eqref{navier-b}, \eqref{basic1} and Poincar\'{e}'s inequality that
\begin{equation}\label{u-h1}
\displaystyle \int_0^T\|u\|_{H^1}^2dt \le C\int_0^T\|\nabla u\|_{L^2}^2dt \le C.
\end{equation}
\end{remark}

Next, we derive the estimates on $L{^\infty}(0,T;L{^p})$-norm of the magnetic field $H$ in 2D case which is a fundamental observation to deal with the coupling and interplay interaction between the fluid motion and the magnetic field.
\begin{lemma}\label{lem-h-lp}
For any $p\geq 2$, there exists a positive constant $C$ such that
\begin{equation}\label{h-lp}
\sup\limits_{0\leq t\leq T}\|H\|_{L^p}^p+\int_0^T\int|H|^{p-2}|\nabla H|^2dxdt\leq C.
\end{equation}
\end{lemma}
\begin{proof}
Multiplying the equation \eqref{CMHD}$_3$ by $p|H|^{p-2}H$ and integrating over $\Omega$, we obtain, by \eqref{boundary} and the Gagliardo-Nirenberg inequality \eqref{g1}, that
\begin{equation}\label{h-lp1}
\begin{split}
&\frac{d}{dt}\int|H|^pdx+\nu p\int|H|^{p-2}|\nabla H|^2dx+\nu p\int\nabla\frac{|H|^2}{2}\cdot\nabla|H|^{p-2}dx \\
&=(1-p)\int|H|^p\mathrm{div} udx+p\int H\cdot \nabla u\cdot |H|^{p-2}Hdx \\
&\leq C\int|H|^p|\nabla u|dx\leq C\|H^p\|_{L^2}\|\nabla u\|_{L^2}\leq C\|H^\frac{p}{2}\|_{L^2}\|\nabla|H|^\frac{p}{2}\|_{L^2}\|\nabla u\|_{L^2} \\
&\leq \frac{\nu p}{2}\int|H|^{p-2}|\nabla H|^2dx+C\|\nabla u\|{_{L^2}^2}\|H\|{_{L^p}^p}.
\end{split}
\end{equation}
which yields that
\begin{equation}
\frac{d}{dt}\|H\|_{L^p}^p+\frac{\nu p}{2}\int|H|^{p-2}|\nabla H|^2dx\leq C\|\nabla u\|{_2^2}\|H\|_{L^p}^p.
\end{equation}
Applying Gronwall's inequality and using \eqref{u-h1}, we have \eqref{h-lp} and finish the proof of Lemma \ref{lem-h-lp}.
\end{proof}

Next, we state a known result concerning the estimate on the $L^\infty(0,T;L^p(\Omega))$-norm of the density  whose proof is similar to that of  \cite{vk1995,jwx2014,mei2015}.
 \begin{lemma}\label{lem-rhop}
Let $\beta >1.$ Then, for any $1\leq p<\infty$, there is a constant $C$ depending on $\Omega$, $T$, $\mu$, $\beta$, $\gamma$, $\nu$, $\|\rho_0\|_{L^\infty}$, $\|u_0\|_{H^1}$ and $\|H_0\|_{H^1}$,
such that
\begin{equation}\label{rhop}
\|\rho\|_{L^{p}}\leq C(T)p^{\frac{2}{\beta-1}}.
\end{equation}
\end{lemma}

In the following, we will use the convention that $C$ denotes a generic positive constant depending on $\Omega$, $T$, $\mu$, $\beta$, $\gamma$, $\nu$, $\|\rho_0\|_{L^\infty}$, $\|u_0\|_{H^1}$ and $\|H_0\|_{H^1}$, and use $C(\alpha)$ to emphasize that $C$ depends on $\alpha$. We follow the ideas of \cite{hl2016,fll2021} for compressible Navier-Stokes equations and obtain the extra integrability up on the momentum $\rho u$ and typical estimates on $\nabla u$ in the following two lemmas, where we modify the proof of \cite{hl2016,fll2021} slightly due to the magnetic field and the boundary effect.
\begin{lemma}\label{lem-rho-u}
There exists some  suitably small generic constant  $\alpha_0\in (0,1)$   which depends only on $\mu$ and $\Omega$ such that
\begin{equation}\label{r-u-al}
\sup_{0\leq t\leq T}\int\rho |u|^{2+\alpha}dx\leq C,
\end{equation}
with
\be\label{al1}\alpha\triangleq R_T^{-\frac{\beta}{2}}\alpha_{0}.\ee
\end{lemma}
\begin{proof}
First, since $u \cdot n=0$ on $\partial \Omega$, there exist a positive constant $\tilde{\alpha}$ depending only on $\Omega$ such that for any $\alpha\in(0,\tilde{\alpha})$,
\begin{equation}\label{tdu-a}
\displaystyle \||u|^{\frac{\alpha}{2}}\nabla u\|_{L^2}\leq C(\||u|^{\frac{\alpha}{2}}\mathrm{div} u\|_{L^2}+\||u|^{\frac{\alpha}{2}}\omega\|_{L^2}).
\end{equation}
Then, following the  proof   of \cite[Lemma 3.7]{hl2016},
multiplying \eqref{CMHD1}$_2$ by $ |u|^\alpha u$ and integrating the resulting equality over $\Omega$ by parts, we arrive at
\begin{equation*}
\begin{split}
&\quad \frac{1}{2+\alpha}\frac{d}{dt}\int\rho |u|^{2+\alpha}dx+ \int|u|^\alpha\big((2\mu+\lambda)(\mathrm{div}u)^2+\mu\omega^2\big)dx\\
&\leq  \alpha\int(2\mu+\lambda)|\mathrm{div}u||u|^{\alpha}|\nabla u|dx+ \alpha\mu\int|u|^{\alpha}|\nabla u|^2dx\\
&\quad+C\int\rho^\gamma|u|^\alpha|\nabla u|dx+\alpha\int|H|^2|u|^{\alpha}|\nabla u|dx\\
&\leq \frac{1}{2}\int|u|^\alpha(2\mu+\lambda)(\mathrm{div}u)^2dx+\frac{\alpha_0^2(2\mu+1)}{2}\||u|^{\frac{\alpha}{2}}\nabla u\|_{L^2}^2+\alpha_0\mu\||u|^{\frac{\alpha}{2}}\nabla u\|_{L^2}^2\\
&\quad+C\int\rho|u|^{2+\alpha}dx +C\int\rho^{\frac{4+2\alpha}{2-\alpha}\gamma-\frac{2\alpha}{2-\alpha}}dx+C\|\nabla u\|_{L^2}^2+C \|H\|_{L^4}^4+C\|H\|_{L^8}^4\|u\|_{H^1}^2
\end{split}
\end{equation*}
which, after choosing $\alpha_0$ suitably small, together with  \eqref{u-h1}, \eqref{h-lp}, \eqref{tdu-a} and Gronwall's inequality yields \eqref{r-u-al} and  finishes the  proof of Lemma \ref{lem-rho-u}.
\end{proof}

\begin{lemma}\label{lem-u-lp}
For $p>2 $ and $\varepsilon\in (0,\frac{\beta}{p}),$ there exists some positive constant $C(p,\varepsilon)$ such that
\begin{equation}\label{tdu-lp}
\|\nabla u\|_{L^{p} }\leq C(p, \varepsilon)R_T^{\frac{1}{2}-\frac{1}{p}+\varepsilon} A_1  \left(\frac{A_2^2}{ A_1^2}\right)^{\frac{1}{2}-\frac{1}{p}}
+C(p, \varepsilon)R_T^{\varepsilon} A_1 .
\end{equation}
\end{lemma}
\begin{proof}
First, we rewrite the momentum equations as
\begin{equation}
\rho\dot{u}=\nabla F+\mu\nabla^\bot\omega-H \cdot \nabla H.  \label{cmhd-2}
\end{equation}
We deduce from \eqref{cmhd-2} and the boundary condition \eqref{navier-b} that $F$ solves the Neumann problem
\begin{equation}\label{de-f}
\begin{cases}
  \Delta F=\mathop{\mathrm{div}}\nolimits(\rho\dot{u}-H \cdot \nabla H)& \mbox{ in } \Omega, \\
   \frac{\partial F}{\partial n}=(\rho\dot{u}-H \cdot \nabla H)\cdot n& \mbox{ on } \partial\Omega.  \end{cases}
\end{equation}
Similarly, $\omega$ solves the related Dirichlet problem:
\begin{equation}\label{de-om}
\begin{cases}
  \mu\Delta\omega=\nabla^\bot\cdot(\rho\dot{u}-H \cdot \nabla H)& \mbox{ in } \Omega, \\
   \omega=0& \mbox{ on } \partial\Omega.  \end{cases}
\end{equation}
Then, standard $L^p$ estimate of elliptic equations (see \cite{NS2004}) implies that for $k\geq 0$ and  $p\in(1, \infty)$,
\begin{equation}\label{f-ow}
\|\nabla F\|_{W^{k, p}}+\|\nabla w\|_{W^{k, p}}\leq C(p, k)(\|\rho\dot{u}\|_{W^{k, p}}+\|H \cdot \nabla H\|_{W^{k, p}}).
\end{equation}
In particular, by \eqref{g1} and \eqref{h-lp}, we have
\bnn
\|\nabla F\|_{L^2}+\|\nabla w\|_{L^2}\leq C(\|\rho\dot{u}\|_{L^2}+\|H \cdot \nabla H\|_{L^2})\leq CR_T^{1/2}A_2+CA_1,
\enn
which together with the Poincar\'{e} inequality, \eqref{h-lp} and \eqref{rhop} yields
\begin{equation} \label{f-om-h1} \ba
\|F\|_{H^1}+\|\omega\|_{H^1}&\leq C(\|\nabla F\|_{L^2}+\|\nabla w\|_{L^2})+\frac{C}{|\Omega|}\int{F}dx\\& \leq C R_T^{1/2}A_2+CA_1 . \ea
\end{equation}
Finally, according to \eqref{tdu1} and \eqref{rhop}, for $\bar{\epsilon}\in (0, 2/p)$, we get
\begin{equation}\label{tdu-lp1}
\begin{split}
\|\nabla u\|_{L^{p} }&\leq C(p)\big(\|\mathrm{div}u\|_{L^{p} }+\|\omega\|_{L^{p} }\big) \\
&\leq C(p) \left\|\frac{F}{2\mu+\lambda}\right\|_{L^{p} }+C(p) \|\omega\|_{L^{p} }+C\|H\|_{L^{2p}}^2+C\|P\|_{L^{p}} \\
&\le C(p)    \left\|\frac{F}{2\mu+\lambda}\right\|_{L^{2} }^{\frac{2}{p}-\bar{\varepsilon}}\|F\|_{L^{\frac{ 2(1+\bar{\varepsilon})p-4}{p\bar{\varepsilon}}} }^{-\frac{2}{p}+1+\bar{\varepsilon}} +C(p)\|\omega\|_{L^{p}} +C(p)
\\&\le  C(p, \bar{\varepsilon})A_1^{\frac{2}{p}-\bar{\varepsilon}}\|F\|_{L^{2} }^{\bar{\varepsilon}} \|F\|_{H^1 }^{1-\frac{2}{p}}+C(p)A_1^{\frac{2}{p}}\| \omega\|_{H^1}^{1-\frac{2}{p}}+C(p)\\
& \leq C(p, \bar{\varepsilon})R_T^{\frac{\beta\bar{\varepsilon}}{2} }A_1^{\frac{2}{p} }  ( \|F\|_{H^1}+\|\omega\|_{H^1})^{1-\frac{2}{p}} +C(p),
\end{split}
\end{equation}
which together with \eqref{f-om-h1} and set $\varepsilon \triangleq \beta\bar{\varepsilon}/2$ gives \eqref{tdu-lp} and finishes the proof of Lemma \ref{lem-u-lp}.
\end{proof}

Now we are in a position to derive the following estimate on the upper bound of $\ln A_1^2$ in terms of $R_T$ which will play an important role in obtaining the upper bound of the density.
\begin{lemma}\label{lem-a1a2}
For any $\kappa\in(0, 1)$,  there is a constant $C(\kappa)$ such that
\begin{equation}\label{a1a2}
\sup_{0\leq t\leq T}\ln A^2_1(t) +\int^T_0 \frac{A^2_2(t)}{A^2_1(t)}dt\leq C(\kappa)R_T^{1 +\kappa\beta }.
\end{equation}
\end{lemma}
\begin{proof}
First, combining \eqref{tdu1}, \eqref{h-lp} and \eqref{rhop} gives
\begin{equation}\label{a1-0}\begin{split}
\|\nabla u\|_{L^2}+\|\nabla H\|_{L^2}&\leq C\|\mathop{\mathrm{div}}\nolimits u\|_{L^2}+C\|\omega\|_{L^2}+C\|\nabla H\|_{L^2}\leq CA_1,
\end{split}
\end{equation}
\begin{equation}\label{f-l2}
\|F\|^2_{L^2}\leq CR_T^{\frac{1}{2}}A_1.	
\end{equation}

Next, direct calculations show that
\begin{equation}
\nabla^\bot\cdot\dot{u} = \frac{D}{Dt}\omega+(\partial_1u\cdot\nabla)u_2-(\partial_2u\cdot\nabla)u_1=\frac{D}{Dt}\omega+\omega \mathrm{div}u,  \label{nbdu}
\end{equation}
and that
\begin{equation}\label{divdu}
\mathrm{div}\dot{u}=\frac{D}{Dt}\bigg(\frac{F}{2\mu+\lambda}\bigg)+\frac{D}{Dt}\bigg(\frac{P }{2\mu+\lambda}\bigg)+\frac{D}{Dt}\bigg(\frac{H^2}{2(2\mu+\lambda)}\bigg)+2\nabla u_1\cdot\nabla^\bot u_2 + (\mathrm{div}u)^2.
\end{equation}
Then multiplying \eqref{cmhd-2} by $2\dot{u}$ and integrating the resulting equality over $\Omega$, by \eqref{nbdu}, \eqref{divdu} and the boundary condition \eqref{navier-b},  we obtain that
\begin{align}\label{a1-1}
&\quad\frac{d}{dt}\Big(\|F/\sqrt{2\mu+\lambda(\rho)}(t)\|_{L^2}^2 +\|\omega(t)\|_{L^2}^2\Big)+2\|\sqrt{\rho}\dot{u}(t)\|_{L^2}^2 \nonumber \\
&=-\frac{d}{dt}\left(\int H \cdot \nabla u \cdot H dx\right)+\int_{\partial\Omega}F \dot{u}\cdot nds-\mu\int\omega^2\mathrm{div}udx-4\int F\nabla u_1\cdot \nabla^\bot u_2dx \nonumber\\
&-2\int F(\mathrm{div}u)^2dx-\int\frac{(\beta-1)\lambda-2\mu}{(2\mu+\lambda)^2}F^2\mathrm{div}udx-\beta\int\frac{\lambda (2P+|H|^2)  }{(2\mu+\lambda)^2}F\mathrm{div}udx \nonumber \\
&+\int\frac{2\gamma P+|H|^2}{ 2\mu+\lambda }F\mathrm{div}udx
 -\int\frac{F}{ 2\mu+\lambda }(\nu \Delta H \cdot H+H \cdot \nabla u \cdot H)dx\nonumber \\
&+\int(H_t \cdot \nabla u \cdot H+H \cdot \nabla u \cdot H_t)dx+\int H \cdot \nabla H \cdot ( u \cdot\nabla u)dx.
\end{align}
Besides, one easily deduces from \eqref{CMHD}$_3$ and \eqref{boundary} that
\begin{align}\label{a1-2}
\displaystyle\frac{d}{dt}\Big(\frac{\nu}{2}\|\nabla H\|_{L^2}^2\Big)_t+\nu^2 \|\Delta H\|_{L^2}^2+\|H_t\|^2_{L^2}\leq \int |H \cdot \nabla u-u \cdot \nabla H-H \mathrm{div} u|^2 dx .
\end{align}
Combine \eqref{a1-1} and \eqref{a1-2}, we have
\begin{align}\label{a1-3}
\quad\frac{d}{dt}A_1^2+A_2^2 \leq &-\frac{d}{dt}\left(\int H \cdot \nabla u \cdot H dx\right)+\int_{\partial\Omega}F u\cdot\nabla u\cdot nds-\mu\int\omega^2\mathrm{div}udx \nonumber\\
&+C\int |F||\nabla u|^2dx+C\int\frac{F^2+P|F|+|H|^2|F|}{2\mu+\lambda}|\mathrm{div}u|dx\nonumber \\
&+C\int\frac{|\Delta H||H|+|H|^2|\nabla u|}{ 2\mu+\lambda }|F|dx\nonumber \\
&+C\int(|H_t||\nabla u||H|+|\nabla H||H||\nabla u||u|)dx \triangleq -\frac{d}{dt}I_0+\sum_{i=1}^6I_i.
\end{align}
Now we estimate each $I_i, i=0,1,\cdot,6$ as follows.
First,  thanks to \eqref{bdd2},  we deal with $I_1$ via:
\begin{equation}\label{a1-i1}\begin{split}
I_1&\leq\left|\int_{\partial\Omega} F(u\cdot \nabla n\cdot u) ds\right|
\leq C\|F\|_{H^1}\|\nabla u\|_{L^2}^2\\
&\leq CR_T^{1/2}A_1^2A_2+CA_1^3\leq \delta A_2^2+CR_TA_1^4.	
\end{split}
\end{equation}
Next, combining \eqref{a1-0}, \eqref{f-om-h1}, and  H\"{o}lder's inequality leads to
\begin{equation}\label{a1-i2}\begin{split}
I_2&\leq C\|\omega\|_{L^4}^2\|\mathrm{div}u\|_{L^2}
\leq C \|\omega\|_{L^2}\|\nabla\omega\|_{L^2}A_1\\
&\leq CR_T^{1/2}A_1^2A_2+CA_1^3\leq \delta A_2^2+CR_TA_1^4. 	
\end{split}
\end{equation}
In the following, for $\kappa\in (0,1),$ letting  $p\ge 2+2/\kappa ,$ we use \eqref{tdu-lp1}, H\"older's and Sobolev's inequalities to get
\begin{equation}\label{a1-i3}
\begin{split}I_3 &
\leq\|F\|_{L^{p}}\|\nabla u\|_{L^{2}}^{\frac{2(p-3)}{p-2}}\|\nabla u\|_{L^{p}}^{\frac{2}{p-2}} \\
&\leq C(p,\bar{\epsilon})\|F\|_{L^{p}}A_1^{\frac{2(p-3)}{p-2}}\left(R_T^{\frac{\bar{\epsilon}\beta}{2} }A_1^{\frac{2}{p} }  ( \|F\|_{H^1}+\|\omega\|_{H^1})^{1-\frac{2}{p}}+1\right)^{\frac{2}{p-2}} 
\\
&\leq C(\kappa) R_T^{\frac{\kappa\beta}{ 2}}\left(\|F||_{H^1} +\|\omega||_{H^1} \right) A_1^2  +C(\kappa)  A_1^2\\
&\leq C(\kappa) R_T^{\frac{1+\kappa\beta}{ 2}}A_1^2A_2+C(\kappa) R_T^{\frac{\kappa\beta}{ 2}}A_1^3+C(\kappa) A_1^2\leq \delta A_2^2+CR_T^{1+\kappa\beta}A_1^4.
\end{split}
\end{equation}
where in the third line, we have used  $ \|F\|_{L^{p}}\le  C(p) \|F\|_{L^{2}}^{\frac{2}{p }}\|F\|_{H^1}^{1-\frac{2}{p }}$  and \eqref{f-l2}.

Next, in terms of  H\"{o}lder's inequality, \eqref{g1} and \eqref{a1-0}, it follows that for $0 <\kappa< 1$,
\begin{equation}\label{f-l2-1}
\begin{split}
\|\frac{F^2}{2\mu+\lambda}\|_{L^2}&=\|\frac{F}{\sqrt{2\mu+\lambda}}\|_{L^4}^2\leq C \big\|\frac{F}{\sqrt{2\mu+\lambda}}\big\|_{L^2}^{1-\kappa} \|F\|^{1+\kappa}_{L^{2(1+\kappa)/\kappa}}\\
&\leq C (\kappa) A_1^{1-\kappa}  \|F \|_{L^2}^{ \kappa} \|F\|_{H^1}
\leq C(\kappa) R_T^{\frac{\kappa\beta}{ 2}}A_1^2 \|F\|_{H^1},
\end{split}
\end{equation}
which together with \eqref{h-lp}, \eqref{rhop} and \eqref{f-om-h1} yields
\begin{equation}\label{a1-i4}
\begin{split}
I_4 &\leq C A_1 \left\|\frac{F^2}{2\mu+\lambda}\right\|_{L^2}+C A_1\|F\|_{L^{2+4\gamma/\beta}}\|P\|_{L^{2+\beta/\gamma}}+CA_1 \|F\|_{L^4}\|H\|_{L^8}^2 \\
&\leq C (\kappa) A_1^{2-\kappa}  \|F \|_{L^2}^{ \kappa} \|F\|_{H^1} +C(\kappa) A_1  \|F\|_{H^1} \\
&
\leq C(\kappa) R_T^{\frac{1+\kappa\beta}{ 2}}A_1^2A_2+C(\kappa) R_T^{\frac{\kappa\beta}{ 2}}A_1^3\leq \delta A_2^2+CR_T^{1+\kappa\beta}A_1^4.
\end{split}
\end{equation}
Similarly,
\begin{equation}\label{a1-i5}\begin{split}
I_5&\leq C\|\frac{F}{2\mu+\lambda(\rho)}\|_{L^4}(\|H\|{_{L^8}^2}\|\nabla u\|_{L^2}+\|H\|_{L^4}\|\Delta H\|_{L^2})\\
&\leq C(\kappa) R_T^{\frac{\kappa\beta}{4}}A_1 \|F\|_{H^1}^{\frac{1}{2}}(A_1+A_2)\leq \delta A_2^2+CR_T^{1+\kappa\beta}A_1^4.	
\end{split}
\end{equation}
By \eqref{h-lp}, \eqref{a1-0} and \eqref{tdu-lp} with $\varepsilon \leq \kappa \beta/4$, we obtain
\begin{equation}\label{a1-i6}\begin{split}
I_6&\leq C (\|H_t\|_{L^2}\|\nabla u\|_{L^4}\|H\|_{L^4}+\|\nabla H\|_{L^4}\|H\|_{L^4}\|\nabla u\|_{L^4}\|u\|_{L^4})\\
&\leq C(\|H_t\|_{L^2}+\|\nabla H\|_{L^2}^{1/2}\|\Delta H\|_{L^2}^{1/2}\|\nabla u\|_{L^2}+\|\nabla H\|_{L^2}\|\nabla u\|_{L^2})\|\nabla u\|_{L^4}\\
&\leq C(p,\epsilon)(A_2+A_1^{3/2}A_2^{1/2}+A_1^2)(R_T^{\frac{1}{4}+\epsilon}A_1^{\frac{1}{2}}A_2^{\frac{1}{2}}+R_T^{\epsilon}A_1)\\
&\leq \delta A_2^2+CR_T^{1+4\epsilon}A_1^4.	
\end{split}
\end{equation}
Moreover, one has
\begin{equation}\label{a1-i0}
I_0\leq\int|H|^2|\nabla u|dx\leq\|\nabla u\|_{L^2}\|H\|{_{L^4}^2}
\leq \delta\|\nabla u\|{_{L^2}^2}+C\|\nabla H\|{_{L^2}^2}.
\end{equation}
Putting all the estimates \eqref{a1-i1}-\eqref{a1-i0} into \eqref{a1-3}, choosing $\delta$ suitably small,  yields that
\begin{equation}\label{a1a2-4}
\displaystyle  \frac{d}{dt}A^2_1(t)+A^2_2(t)\leq C(\kappa)R_T^{1+ \kappa\beta} A_1^4.
\end{equation}
Besides, combining \eqref{tdu1}, \eqref{h-lp} and \eqref{rhop} gives
\begin{equation}\label{a1}\begin{split}
A_1^2(t) \leq& C\|\sqrt{2\mu+\lambda}(\mathop{\mathrm{div}}\nolimits u)\|^2_{L^2}+C\|P\|^2_{L^2}+C\|H\|^4_{L^4}+C\|\omega\|^2_{L^2}+C\|\nabla H\|^2_{L^2}  \\
\leq& C\|\sqrt{2\mu+\lambda}(\mathop{\mathrm{div}}\nolimits u)\|^2_{L^2}+C\|\omega\|^2_{L^2}+C\|\nabla H\|^2_{L^2}+C, 
\end{split}
\end{equation}
which together with \eqref{basic1} implies
\begin{equation}\label{a1-l2}
\int_0^T A_1^2 dt \le C.
\end{equation}
Combining \eqref{a1a2-4} with \eqref{a1-l2}, applying Gronwall's inequality, yields \eqref{a1a2} and finishes the proof of Lemma \ref{lem-a1a2}.
\end{proof}

In the following, adapting the ideas in \cite{fll2021}, we derive the pointwise representation of $F$ by applying Riemann mapping theorem and  the pull-back Green's function method.
Precisely, let  $\varphi=(\varphi_1, \varphi_2):\bar{\Omega}\rightarrow\bar{\mathbb{D}}$ be the conformal mapping  which  satisfies  Lemma \ref{lem-varphi}.
We then define the pull back Green's function $N$ of $\Omega$:
\begin{equation}\ba\label{n-fun}
 N (x, y)\triangleq -\frac{1}{2\pi}\Big(\ln|\varphi(x)-\varphi(y)|+ \ln\big||\varphi(x)|\varphi(y)-\frac{\varphi(x)}{|\varphi(x)|}\big| \Big).\ea
\end{equation}
Since the outer normal derivative of $ N $ on the boundary $\partial\Omega$ is no longer constant, $N$ is   not the ``real" Green's function of $\Omega$ in the classical sense  but still sufficient for our further calculations. Moreover, since the conformal mapping preserves angles, one immediately has the following conclusion (see \cite[Lemma 3.6]{fll2021}):
\begin{equation}\label{lem-no}
\frac{\partial{N}}{\partial n}(x, y_0)=-\frac{1}{2\pi}|\nabla\varphi_1(y_0)|, 	
 \end{equation}
where $n$ is the unit outer normal at $y_0\in\partial\Omega$.

Now, we turn to use the pull back Green's function $ N $ defined as in \eqref{n-fun}  to give a pointwise representation of $F$ in $\Omega$ via Green's identity as follows:
\begin{lemma}\label{lem-f}
Let $F\in C^1(\bar\Omega)\cap C^2(\Omega)$ solve the Neumann problem
\begin{equation}\label{de-ff}
\begin{cases}
  \Delta F=\mathop{\mathrm{div}}\nolimits(\rho\dot{u}-H \cdot \nabla H)& \mbox{ in } \Omega, \\
   \frac{\partial F}{\partial n}=(\rho\dot{u}-H \cdot \nabla H)\cdot n& \mbox{ on } \partial\Omega,  \end{cases}
\end{equation}
and $ N $ is the pull back Green's function defined as in \eqref{n-fun}. Then for $x\in\Omega$, there holds
\begin{equation} \label{f}
F(x)
=  -\int \nabla_y N (x, y)\cdot(\rho\dot{u}+H\cdot\nabla H)(y)dy+\int_{\partial\Omega}\frac{\partial  N }{\partial n}(x, y)F(y)dS_y.
\end{equation}
Furthermore, for the boundary condition $u \cdot n=0$ on $\partial \Omega$, it follows that
\begin{equation}\label{f-123}
F(x)=-\frac{D}{Dt}F_0+F_1+F_2+F_3,
\end{equation}
with
\begin{align}
F_0=&-\int\partial_{y_j} N (x, y)\rho u_j(y)dy,\label{f0}\\
F_1=&\int \left[\partial_{x_i}\partial_{y_j} N (x, y)u_i(x)+\partial_{y_i}\partial_{y_j} N (x, y)u_i(y)\right]\rho u_j(y)dy,\label{f1}\\
F_2=&-\int \nabla_y N (x, y)\cdot(H\cdot\nabla H)(y)dy,\label{f2}\\
F_3=&\int_{\partial\Omega}\frac{\partial  N }{\partial n}(x, y)F(y)dS_y.\label{f3}
\end{align}
\end{lemma}
\begin{proof}
The proof of \eqref{de-ff} is essentially the same as that in \cite[Lemma 3.7]{fll2021} and we just sketch it for completeness.
Denote $ N (x, y)=N_1(x, y)+N_2(x, y)$,  where
\begin{equation*}
N_1(x, y)=-\frac{1}{2\pi}\ln |\varphi(x)-\varphi(y)|,  \ N_2(x, y)=-\frac{1}{2\pi}\ln\big||\varphi(x)|\varphi(y)-\frac{\varphi(x)}{|\varphi(x)|}\big|.
\end{equation*}
By \eqref{harm}, $N_2(x, \cdot)$ remains harmonic in $\Omega$. Applying Green's second identity to $F$ and $N_2(x, \cdot)$ in $\Omega$, we have
\begin{equation}\label{fn2}
-\int_{}N_2\Delta Fdy=\int_{\partial\Omega}\left(\frac{\partial N_2}{\partial n}F-N_2\frac{\partial F}{\partial n}\right)dS_y.
\end{equation}
For $N_1(x, \cdot)$ in $\Omega\backslash \widetilde{B}_r(x)$, where $\widetilde{B}_r(x)=\varphi^{-1}(B_r(\varphi(x)))$ is just the inverse image under $\varphi$ of the ball centered at $\varphi(x)$ of small enough radius $r$, we deduce that
\begin{equation}\nonumber
\ba&
\int_{\Omega\backslash \widetilde{B}_r(x)}(\Delta N_1F-N_1\Delta F)dy\\&=\int_{\partial\Omega}\left(\frac{\partial N_1}{\partial n}F-N_1\frac{\partial F}{\partial n}\right)dS_y
-\int_{\partial\widetilde{B}_r(x)}\left(\frac{\partial N_1}{\partial n}F-N_1\frac{\partial F}{\partial n}\right)dS_y.
\ea
\end{equation} Letting $r\rightarrow 0$ and using the fact that
  $\Delta N_1=0$ in $\Omega\backslash \widetilde{B}_r(x), $  we have
\begin{equation} \label{fn1}\ba
-\int_{\Omega} N_1\Delta Fdy=&\int_{\partial\Omega}\left(\frac{\partial N_1}{\partial n}F-N_1\frac{\partial F}{\partial n}\right)dS_y\\&-\lim_{r\rightarrow0}\int_{\partial\widetilde{B}_r(x)}\left(\frac{\partial N_1}{\partial n}F-N_1\frac{\partial F}{\partial n}\right)dS_y.
\ea\end{equation}
For $r$ small enough,
\begin{equation}\label{fn1-2}
\begin{split}
&\quad\bigg|\int_{\partial\widetilde{B}_r(x)}N_1\frac{\partial F}{\partial n}dS_y\bigg|
\leq\int_{\partial\widetilde{B}_r(x)}|N_1(x, y)|dS_y\cdot \sup_{\partial\widetilde{B}_r(x)}|\nabla F|\\
&\leq\int_{\partial B_r(\varphi(x))}\left|\mathrm{ln} {|\varphi(x)-\varphi(y)|}\right| \frac{1}{|\nabla\varphi_1(y)|}dS_{\varphi(y)} \cdot\sup_{\partial\widetilde{B}_r(x)}|\nabla F|\\
&\leq Cr|\mathrm{ln}\,r|\cdot\sup_{\partial\widetilde{B}_r(x)} |\nabla\varphi_1|^{-1}\sup_{\partial\widetilde{B}_r(x)}|\nabla F| \rightarrow 0\ as\ r\rightarrow 0.	
\end{split}
\end{equation}
and
\begin{equation}\label{fn1-3}
\begin{split}
&\quad \int_{\partial\widetilde{B}_r(x)}F\frac{\partial N_1}{\partial n}dS_y
=\int_{\partial\widetilde{B}_r(x)}\vec{n}_y\cdot\nabla_yN_1(x, y) F(y)dS_y\\&
= -\frac{1}{2\pi  }\int_{\partial\widetilde{B}_r(x)} \frac{\nabla\varphi(y)[\varphi(y)-\varphi(x)]}{|\nabla\varphi_1(y)| |\varphi(x)-\varphi(y)|} \cdot \frac{\nabla\varphi(y)[\varphi(x)-\varphi(y)]}{|\varphi(x)-\varphi(y)|^2} F(y)dS_y\\
&=-\frac{1}{2\pi r}\int_{\partial\widetilde{B}_r(x)}|\nabla\varphi_1(y)|F(y)dS_y\\
&=-\frac{1}{2\pi r}\int_{\partial B_r(\varphi(x))}F(\varphi^{-1}(\tilde{y}))dS_{\tilde{y}}\rightarrow F(x)\ as \ r\rightarrow 0,	
\end{split}
\end{equation}
where we use \eqref{c-r} and \eqref{lem-no}.  Adding \eqref{fn2} and \eqref{fn1} together, by using \eqref{de-ff}, \eqref{fn1-2} and \eqref{fn1-3}, we have
\begin{equation}\nonumber
\begin{split}
F(x)=&\int  N (x, y)\mathrm{div}(\rho\dot{u}+H\cdot\nabla H)(y)dy-\int_{\partial\Omega} N (x, y)(\rho\dot{u}+H\cdot\nabla H)\cdot\vec{n}(y)dS_y\\
&+\int_{\partial\Omega}\frac{\partial  N }{\partial n}(x, y)F(y)dS_y\\\label{01}
=&-\int \nabla_y N (x, y)\cdot(\rho\dot{u}+H\cdot\nabla H)(y)dy +\int_{\partial\Omega}\frac{\partial  N }{\partial n}(x, y)F(y)dS_y,
\end{split}
\end{equation} which gives \eqref{f}.
By \eqref{CMHD}$_1$ and the boundary condition $u\cdot {n}|_{\partial \Omega}=0$, it follows that
\begin{equation}\label{ff-n1}
\begin{split}
&-\int \nabla_y N (x, y)\cdot\rho\dot{u}(y)dy\\
=&-\int \nabla_y N (x, y)\cdot\bigg(\frac{\partial(\rho u)}{\partial t}+\mathrm{div}(\rho u\otimes u)\bigg)dy\\
=&-(\frac{\partial}{\partial t}+u\cdot\nabla)\int \partial_{y_j} N (x, y)\rho u_j(y)dy\\
&+\int \left[\partial_{x_i}\partial_{y_j} N (x, y)u_i(x)+\partial_{y_i}\partial_{y_j} N (x, y)u_i(y)\right]\rho u_j(y)dy.	
\end{split}
\end{equation}
Then we rewrite \eqref{f} as follows
\begin{equation} \label{f-1}
\begin{split}
F(x)=&-\frac{D}{Dt}\int\partial_{y_j} N (x, y)\rho u_j(y)dy\\&+\int \left[\partial_{x_i}\partial_{y_j} N (x, y)u_i(x)+\partial_{y_i}\partial_{y_j} N (x, y)u_i(y)\right]\rho u_j(y)dy\\
&-\int \nabla_y N (x, y)\cdot(H\cdot\nabla H)(y)dy+\int_{\partial\Omega}\frac{\partial  N }{\partial n}(x, y)F(y)dS_y.	
\end{split}
\end{equation}
and finishes the proof of Lemma \ref{lem-f}.
\end{proof}

To this end, we are in a position to obtain the upper bound of the density which plays an essential role in the whole procedure.
\begin{proposition} \label{lem-rho-inf} Assume that \eqref{b-g} holds. Then there exists some positive constant  $C$ such that
\begin{equation}\label{rho-inf1}
\sup_{0\leq t\leq T}\|\rho\|_{L^{\infty}}\leq C .
\end{equation}
\end{proposition}
\begin{proof}
Denote $\theta(\rho)\triangleq 2\mu \ln\rho+\beta^{-1}\rho^\beta,$  we have by \eqref{CMHD}$_1$ and \eqref{f-123},
\begin{equation}\label{dt-rho}\begin{split}
  \frac{D}{Dt}\theta(\rho) =&-(2\mu+\lambda)\mathrm{div} u = -F-P-|H|^2/2\leq -F\\
  \le& \frac{D}{Dt}F_0+\sum_{i=1}^3|F_i|.
\end{split}\end{equation}

First, we state the crucial point-wise estimate on $F_1$ due to \cite[Proposition 3.2]{fll2021}. Precisely,  for any $x\in \Omega$ with $\varphi(x)\not=0$, there exists a generic positive constant $C(\Omega)$ such that
\begin{equation}\label{f1-1}\begin{split}
|F_1(x)|\leq& C(\Omega) \int\frac{\rho|u|^2(y)}{|x-y|}dy + C(\Omega)\int\frac{|u(x)-u(y)|}{|x-y|^2}\rho|u|(y)dy \\& + C(\Omega)\int\frac{|u(x')-u(y)|} {|x'-y|^2}\rho|u|(y)dy \triangleq C(\Omega)(F_{c1}+F_{c2}+F_{c3}), 	
\end{split}\end{equation}
where $x'\triangleq \varphi^{-1} \left( \frac{\varphi(x)}{|\varphi(x)|}\right)\in \partial\Omega$.

As mentioned in \cite{fll2021}, to derive the precise control of $F_1$, we need to focus on the commutator-tpye terms $F_{c2}$ and $F_{c3}$, which represent the different between the periodic domains and general bounded domains.
In fact, for $F_{c1}$, \eqref{a1-0} and Sobolev's inequality leads to
\begin{equation}\label{f1-e1}
\begin{split}
 F_{c1} \leq R_T \left(\int |x-y|^{-\frac{3}{2}}dy \right)^{\frac{2}{3}}
\left(\int|u|^6 dy\right)^{\frac{1}{3}}\leq CR_T \|u\|_{H^1}^2\leq CR_T A_1^2.
\end{split}
\end{equation}
It remains to deal with $F_{c2}$ and $F_{c3}$.
Since the two commutator-tpye terms are similar, we merely estimate $F_{c2}$ by the similar method as that used in \cite{fll2021}.
By Sobolev's embedding theorem \cite[Chapter 5]{evans2010}, for $p\in(2,4)$ which will be determined later, we have for any $x, y\in \overline{\Omega}, $
\begin{equation*}
|u(x)-u(y)|\leq C(p )\|\nabla u\|_{L^p}|x-y|^{1-\frac{2}{p}},\quad \text{for any} \ \ x, y\in \overline{\Omega},
\end{equation*}
which implies
\begin{equation}\label{ff-2}
\ba
F_{c2} &\leq C(p) \int\frac{\|\nabla u\|_{L^p}\cdot|x-y|^{1-\frac{2}{p}}}{|x-y|^2}\rho|u|(y)dy\\&
=C(p)\|\nabla u\|_{L^p}\int |x-y|^{-(1+\frac{2}{p})} \rho|u|(y)dy.
\ea
\end{equation}
Then, for $\delta>0 $ and $\varepsilon_0\in (0, (p-2)/8)  $   which will be determined later,  on the one  hand,  we use  \eqref{g1} and \eqref{a1-0}  to get
\begin{equation}\label{ff-3}
\ba
 &
\int_{|x-y|<2\delta}|x-y|^{-\left(1+\frac{2}{p}\right)}\rho|u|(y)dy\\
&\leq C(p)R_T
\left(\int_{|x-y|<2\delta} |x-y|^{-\left(1+\frac{2}{p}\right) (1+\varepsilon_0)} dy\right)^{\frac{1}{1+\varepsilon_0}}
\|u\|_{L^{\frac{1+\varepsilon_0}{\varepsilon_0}}}\\
&\leq C(p)  R_T
\delta^{1-\frac{2}{p}-\frac{2\varepsilon_0}{1+\varepsilon_0}} \left(\frac{1+\varepsilon_0}{\varepsilon_0}\right)^{\frac{1}{2}}\|u\|_{H^1}\\
&\leq C(p)   R_T \varepsilon_0^{-\frac{1}{2}} A_1
\delta^{1-\frac{2}{p}-\frac{2\varepsilon_0}{1+\varepsilon_0}} .
\ea
\end{equation}
On the other hand,  for $\alpha=R_T^{-\frac{\beta}{2}}\alpha_{0}$ as in \eqref{al1},  we use  Lemma \ref{lem-rho-u} to derive
\begin{equation}
\ba
  &\int_{|x-y|>\delta}|x-y|^{-\left(1+\frac{2}{p}\right)}\rho|u|(y)dy\\
&\leq
C(p)
\left(\int_{|x-y|>\delta} |x-y|^{-\left(1+\frac{2}{p}\right) (\frac{2+\alpha}{1+\alpha})} dy\right)^{\frac{1+\alpha}{2+\alpha}}
\left(\int\rho^{2+\alpha}|u|^{2+\alpha}dx\right)^{\frac{1}{2+\alpha}}\\
&\leq C(p)R_T^{ \frac{1+\alpha}{2+\alpha}}
 \delta^{-\frac{2}{p}+\frac{\alpha}{2+\alpha}}.
\ea
\end{equation}
Now, we choose $\delta >0$ such that
\be \delta^{-\frac{2}{p}+\frac{\alpha}{2+\alpha}} =
A_1^{ \frac{2}{p} },\ee which in particular implies
\be \label{ff-4} A_1\delta^{1-\frac{2}{p}-\frac{2\varepsilon_0}{1+\varepsilon_0}} =A_1^{ \frac{2}{p} },\ee
provided we set \be\label{ff-5} \varepsilon_0=\frac{(p-2)\alpha}{8+(6-p) \alpha}\in \left(0, \frac{p-2}{8}\right). \ee
Then, it follows from \eqref{ff-3}-\eqref{ff-4} that
\begin{equation*}
\ba\label{ff-6}&
 \int |x-y|^{-\left(1+\frac{2}{p}\right)}\rho|u|(y)dy \\&\le \left( \int_{|x-y|<2\delta} +\int_{|x-y|> \delta}\right)|x-y|^{-\left(1+\frac{2}{p}\right)}\rho|u|(y)dy \\
&\leq C(p)  R_T\varepsilon_0^{-1/2}   A_1^{\frac{2}{p}}+C(p)R_T^{ \frac{1+\alpha}{2+\alpha}}A_1^{\frac{2}{p}}\\
&\leq C(p)  R_T^{1+\beta/4}   A_1^{\frac{2}{p}} ,
\ea
\end{equation*}
where in the last line we have used $\varepsilon_0^{-1/2}\le C(p)\alpha^{-1/2}\le C(p)R_T^{\beta/4} $ due to \eqref{ff-5}. Combining this,  \eqref{tdu-lp},  and
 \eqref{ff-2} shows that for any $\varepsilon>0,$ choosing $p=2/(1-\varepsilon),$
\begin{equation}\label{f1-e2}
\begin{split}
 F_{c2}&\leq C(p)\|\nabla u\|_{L^p}R_T^{1+\frac{\beta}{4}}A_1^{\frac{2}{p}} \\
&\leq C(p , \varepsilon )R_T^{\frac{3}{2}-\frac{1}{p}+\varepsilon+\frac{\beta}{4}}A_1^{1+\frac{2}{p}} \left(\frac{A_2^2}{ A_1^2}\right)^{\frac{1}{2}-\frac{1}{p}}
+C(p, \varepsilon)R_T^{1+\varepsilon+\frac{\beta}{4}}A_1^{1+\frac{2}{p}}\\
&\leq C(\varepsilon )R_T^{1+\frac{\beta}{4}+\frac{3\varepsilon}{2}}A_1^{2-\varepsilon}\left(\frac{A_2^2}{A_1^2}\right)^{\frac{\varepsilon}{2}}+C(\varepsilon)R_T^{1+\frac{\beta}{4}+\varepsilon}A_1^{2-\varepsilon}.
\end{split}
\end{equation}
where in the last line we have used  \eqref{u-h1} and \eqref{a1a2}.
Combining \eqref{f1-e1} and \eqref{f1-e2} with \eqref{f1-1} yields
\begin{equation}\label{f1-2}\begin{split}
\max_{x\in \overline\Omega}|F_1(x)|\leq & C(\Omega,\varepsilon )R_T^{1+\frac{\beta}{4}+\frac{3\varepsilon}{2}}A_1^{2-\varepsilon}\left(\frac{A_2^2}{A_1^2}\right)^{\frac{\varepsilon}{2}}+C(\Omega,\varepsilon )R_T^{1+\frac{\beta}{4}+\varepsilon}A_1^{2}.
\end{split}\end{equation}
Then, integrating \eqref{f1-2} with respect to $t$ and using the H\"{o}lder's inequality, \eqref{a1a2} and \eqref{a1-l2} with $\kappa\leq 2/\beta$, we obtain that
\begin{equation}\label{f1-3}\begin{split}
&\int_0^T \max_{x\in \overline\Omega}|F_1(x)|dt \\
\leq & C(\varepsilon )R_T^{1+\frac{\beta}{4}+\frac{3\varepsilon}{2}}\left(\int_0^T A_1^2dt\right)^{\frac{2-\varepsilon}{2}}\left(\int_0^T\frac{A_2^2}{A_1^2}dt\right)^{\frac{\varepsilon}{2}}+C(\varepsilon )R_T^{1+\frac{\beta}{4}+\varepsilon}\int_0^T A_1^2dt\\
\leq & C(\varepsilon )R_T^{1+\frac{\beta}{4}+3\varepsilon}.
\end{split}\end{equation}

Next, for $F_2$, in virtue of \eqref{h-lp}, it follows that
\begin{equation}\label{f2-1}
\begin{split}
|F_2| &\leq C \int|x-y|^{-1}|H\cdot\nabla H(y)|dy\\
&\leq C\left(\int |x-y|^{-\frac{3}{2}}dy \right)^{\frac{2}{3}}
\left(\int|H\cdot\nabla H(y)|^3 dy\right)^{\frac{1}{3}}\\
&\leq C \|H\|_{L^{12}}\|\nabla H\|_{L^{4}}\leq C\left(\frac{A_2^2}{ A_1^2}\right)^{\frac{1}{4}}A_1+ CA_1,
\end{split}
\end{equation}
which together with \eqref{a1a2} and \eqref{a1-l2} yields
\begin{equation}\label{f2-2}
\begin{split}
\int_0^T |F_2|dt
\leq C \int_0^T\left(\frac{A_2^2}{ A_1^2}\right)^{\frac{1}{4}}A_1dt+ C \int_0^TA_1dt\leq CR_T^{\frac{1}{4}+\frac{\kappa\beta}{4}}.
\end{split}
\end{equation}

Similarly, for the boundary term $F_3$, by \eqref{f-om-h1}, \eqref{a1a2} and \eqref{lem-no}, we have
\begin{equation}\label{f3-2}
\begin{split}
\int_0^T |F_3|dt& \leq C \int_0^T\|F\|_{H^1}dt
\leq CR_T^{\frac{1}{2}}\int_0^T\left(\frac{A_2^2}{ A_1^2}\right)^{\frac{1}{2}}A_1 dt+ C\leq C (\varepsilon)R_T^{1+\varepsilon}.
\end{split}
\end{equation}

Finally, for $F_0$, a direct computation yields that for $\alpha=R_T^{-\frac{\beta}{2}}\alpha_{0}$ as in Lemma \ref{r-u-al},
\begin{equation}\label{f0-1}
\begin{split}
|F_0|&\le C\int |x-y|^{-1}\rho(y)|u(y)|dy\\
&\le C\left(\int|x-y|^{-\frac{2+\alpha}{1+\alpha}} dy\right)^{\frac{1+\alpha}{2+\alpha}} \left(\int\rho^{2+\alpha}|u|^{2+\alpha}dy\right)^{\frac{1 }{2+\alpha}}\\&\le C\alpha^{-\frac{1+\alpha}{2+\alpha}}R_T^{\frac{1+\alpha}{2+\alpha}}\left(\int\rho |u|^{2+\alpha}dy\right)^{\frac{1 }{2+\alpha}} \\&\le C R_T^{\left(1+\frac{\beta}{2}\right)\frac{1+\alpha}{2+\alpha}}\le C R_T^{\frac{2+\beta}{3}}. 	
\end{split}
\end{equation}
Integrating \eqref{dt-rho} with respect to $t,$ we obtain after using  \eqref{f1-3}, \eqref{f2-2}, \eqref{f3-2}  and \eqref{f0-1} that
\bnn R_T^\beta\le C(\varepsilon) R_T^{\max\left\{1+\frac{\beta}{4}+3\varepsilon, {\frac{2+\beta}{3}}\right\}} .\enn
Since $\beta>4/3,$ this  in particular implies
\begin{equation}\label{rho-inf}
\sup_{0\leq t\leq T}\|\rho\|_{L^\infty}\leq C,
\end{equation}
with $\varepsilon$ suitably small and finishes the proof.
\end{proof}

\section{A priori estimates (II): lower and higher order ones}\label{se4}
In this section, we will proceed to study the lower and high order estimates based on the previous estimates. Here we adopt the method of the article \cite{hl2016,fll2021}, and mainly focus on the magnetic field and the boundary terms. We sketch it here for completeness.
We always assume that $(\rho_0,u_0,H_0)$ satisfies \eqref{dt2} and $(\rho,u,H)$ is the strong solution to \eqref{CMHD}-\eqref{boundary} on $\Omega\times(0,T]$ obtained by Lemma \eqref{lem-local}.

\begin{lemma} \label{lem-low1} There exists some positive constant $C$ depending only on $\Omega$, $T$, $\mu$, $\beta$, $\gamma$, $\nu$, $\|\rho_0\|_{L^\infty}$, $\|u_0\|_{H^1}$ and $\|H_0\|_{H^1}$ such that
\begin{equation}\label{low1}
\sup_{0\leq t\leq T}(\|u\|_{H^1}+\|H\|_{H^1})+\int_{0}^{T}\left( \|\omega\|_{H^{1}}^{2}+\|F\|_{H^{1}}^{2}+A_2^2\right)dt\leq C .
\end{equation}
\end{lemma}
\begin{proof}
First, \eqref{a1a2} together with \eqref{a1-0}, \eqref{rho-inf} and Poincar\'{e} inequality gives
\begin{equation}
\displaystyle \sup_{0\leq t\leq T}(\|u\|_{H^1}+\|H\|_{H^1})\leq C\sup_{0\leq t\leq T}A_1\leq C.
\end{equation}
Next, by using the Gronwall's inequality, \eqref{a1a2-4} with \eqref{f-om-h1}, \eqref{rho-inf} yields
 \begin{equation}
\int_{0}^{T}\left(A_2^2+\|\omega\|_{H^{1}}^{2}+\|F\|_{H^{1}}^{2}\right)dt\leq C,
\end{equation}
and finishes the proof of Proposition \ref{lem-rho-inf}.
\end{proof}

\begin{lemma}\label{lem-a0}
There exists some positive constant $C$ depending only on $\Omega$, $T$, $\mu$, $\beta$, $\gamma$, $\nu$, $\|\rho_0\|_{L^\infty}$, $\|u_0\|_{H^1}$ and $\|H_0\|_{H^1}$ such that
\begin{align}\label{a02}
 \displaystyle  \sup_{0\leq t\leq T} \sigma A_2^2+\int_0^T \sigma(\|\nabla\dot{u}\|_{L^2}^2+\|\nabla H_t\|_{L^2}^2)dt \le   C,
\end{align}
with $\si=\si(t)\triangleq\min\{1,t \}$.
Moreover,  for any $p\in[1, \infty)$,  there is a positive constant $C(p)$ depending only on $p$, $\Omega$, $T$, $\mu$, $\beta$, $\gamma$, $\nu$, $\|\rho_0\|_{L^\infty}$, $\|u_0\|_{H^1}$ and $\|H_0\|_{H^1}$ such that
\begin{equation}\label{tdu-lp-t}
\sup_{0\leq t\leq T} \sigma\|\nabla u\|^2_{L^p}
\leq C(p).
\end{equation}
\end{lemma}
\begin{proof} Operating $\dot{u}_{j}[\pa/\pa t+\mathrm{div} (u\cdot)] $ to $ (\ref{cmhd-2})_j,$ summing with respect to $j$, and integrating over $\Omega,$ together with $ \eqref{CMHD}_1 $, we get
\begin{align}\label{J0}
\displaystyle \frac{1}{2}\frac{d}{dt}\left(\int\rho|\dot{u}|^{2}dx\right)
& = \int (\dot{u}\cdot\nabla F_t+\dot{u}_{j}\,\mathrm{div}(u\partial_jF))dx \nonumber \\
&\quad+\mu\int (\dot{u}\cdot\nabla^{\bot}\omega_t+\dot{u}_{j}\mathrm{div}((\nabla^{\bot}\omega)_j\,u))dx \nonumber \\
&\quad+\int (\dot{u}\cdot(\mathrm{div} (H\otimes H))_t+\dot{u}_{j}\mathrm{div}((\mathrm{div} (H\otimes H_j)\,u))dx \nonumber \\
& \triangleq J_1+ J_2+ J_3.
\end{align}
Let us estimate $J_1, J_2$ and $J_3$.
By \eqref{navier-b} and \eqref{CMHD1}$_1$, a direct computation yields
\begin{equation}\label{J10}
\begin{split}
J_1=&\int(\dot{u}\cdot\nabla \dot{F}+\dot{u}_j\partial_jF\mathrm{div}u-\dot{u}\cdot\nabla u\cdot\nabla F)dx\\
=&\int_{\partial\Omega}\dot{F}(\dot{u}\cdot n)ds-\int\dot{F}\mathrm{div}\dot{u}dx+\int(\dot{u}_j\partial_jF\mathrm{div}u-\dot{u}\cdot\nabla u\cdot\nabla F)dx\\
=&\int_{\partial\Omega}\dot{F}(\dot{u}\cdot n)ds-\int(2\mu+\lambda)(\mathrm{div}\dot{u})^2dx+\int(2\mu+\lambda)\mathrm{div}\dot{u}\nabla u:(\nabla u)^Tdx \\
&+\int\beta\lambda(\mathrm{div}u)^2 \mathrm{div}\dot{u}dx
+\int \mathrm{div}\dot{u}\,H \cdot H_tdx+\int \mathrm{div}\dot{u}\,u \cdot \nabla H \cdot H dx \\
&+\gamma\int P\mathrm{div}u\mathrm{div}\dot{u} dx+\int(\dot{u}_j\partial_jF\mathrm{div}u+\dot{u}\cdot\nabla u\cdot\nabla F)dx \\
\le &\int_{\partial\Omega} F_t (\dot{u}\cdot n)ds+\int_{\partial\Omega} u\cdot\nabla F (\dot{u}\cdot n)ds-\mu\int (\mathrm{div}\dot{u})^2dx+C\|\nabla u\|_{L^4}^4 \\
&+C+\frac{\delta}{2}\|\nabla H_t\|_{L^2}^2+C\|H_t\|_{L^2}^2+C\|\Delta H\|_{L^2}^2+C \int |\dot{u}| |\nabla F||\nabla u|dx,
\end{split}
\end{equation}
where in the second equality we have used
\begin{align*}
\dot{F}=(2\mu+\lm)\mathrm{div}\dot u-\beta\lm(\mathrm{div}u)^2-(2\mu+\lm)\na u:\na u+u \cdot \nabla H \cdot H+\ga P\mathrm{div} u-H \cdot H_t.
\end{align*}
and
\begin{align}
&\|H \cdot H_t\|_{L^2}^2\leq C\|H\|_{L^4}^2\|H_t\|_{L^4}^2\leq \delta\|\nabla H_t\|_{L^2}^2+C\|H_t\|_{L^2}^2+C,\\
&\|u \cdot \nabla H \cdot H\|_{L^2}^2\leq C\|H\|_{L^8}^2\|u\|_{L^8}^2\|\nabla H\|_{L^4}^2\leq C\|\Delta H\|_{L^2}^2+C.
\end{align}
It is necessary to estimate the two boundary terms in the last inequality and using the observations \eqref{bdd2} and \eqref{bdd3}. For the first term on the righthand side of \eqref{J10}, we have
\begin{equation}\label{J11}\begin{split}
&\int_{\partial\Omega} F_t\dot{u}\cdot nds=-\int_{\partial\Omega} F_t\,(u\cdot\nabla n\cdot u)ds \\
= & -\frac{d}{dt}\left(\int_{\partial\Omega} (u\cdot\nabla n\cdot u)Fds\right)+\int_{\partial\Omega} \big(F\dot{u}\cdot\nabla n \cdot u+Fu\cdot\nabla n \cdot\dot{u}\big)ds \\
&\quad  -\int_{\partial\Omega} \big(F(u \cdot \nabla) u\cdot\nabla n \cdot u\big)ds-\int_{\partial\Omega} \big(Fu\cdot\nabla n \cdot(u \cdot \nabla) u\big)ds\\
\triangleq & -\frac{d}{dt}\left(J_{b0}\right)+J_{b1}+J_{b2}+J_{b3}.	
\end{split}\end{equation}
Note that $J_{b0}$ has been estimated in \eqref{a1-i1}, it only remains to estimate $J_{bi}, i=1,2,3$.
By \eqref{udot}, \eqref{f-om-h1}, \eqref{rho-inf1} and \eqref{low1}, we obtain
\begin{equation}\label{J11-1}\begin{split}
J_{b1}\leq &C\|  F\|_{H^1}\| u\|_{H^1}\| \dot{u}\|_{H^1}
\leq C(\|\nabla\dot{u}\|_{L^2}+C)(A_2+C)\\
\leq&\frac{\delta}{6} \|\nabla\dot{u}\|_{L^2}^2+CA_2^2+C.
\end{split}\end{equation}
For $J_{b2}$, using $\eqref{bdd3}$, \eqref{udot}, \eqref{f-om-h1}, \eqref{rho-inf1} and \eqref{low1} yields
\begin{equation}\label{J11-2}
\begin{split}
&|J_{b2}|=\left|\int_{\partial\Omega}F((u\cdot\nabla u)\cdot\nabla n\cdot u)ds \right|
=\left|\int_{\partial\Omega}F(u\cdot n^\bot)n^\bot \cdot\nabla u_i\partial_i n_j u_jds \right|\\
=&\left|\int_{\Omega}\nabla^\bot \cdot\left(\nabla u_i\partial_i n_j u_jF(u\cdot n^\bot)\right)dx \right|
=\left|\int_{\Omega}\nabla u_i\cdot\nabla^\bot \left(\partial_i n_j u_jF(u\cdot n^\bot)\right)dx \right|\\ \le& C\int_\Omega |\nabla u|\left(|F||u|^2+|F||u||\nabla u|+|u|^2|\nabla F|\right)dx \\ \le & C\|\nabla u \|_{L^4}\left(\|F\|_{L^4}\|u\|_{L^4}^2+\|F\|_{L^4}\|u\|_{L^4}\|\nabla u\|_{L^4}+\|\nabla F\|_{L^2}\|u\|_{L^8}^2\right)\\ \le & C\|\nabla u \|_{L^4}^4+C \|F\|_{H^1}^2+C \leq C\|\nabla u\|_{L^4}^4+CA_2^2+C.
\end{split}
\end{equation}
Similarly, for $J_{b3}$, we also have
\begin{equation}
\displaystyle |J_{b3}|\leq C\|\nabla u\|_{L^4}^4+CA_2^2+C.
\end{equation}
which together with \eqref{J11}-\eqref{J11-2} leads to
\begin{equation}\label{J11-3}\begin{split}
\int_{\partial\Omega} F_t\dot{u}\cdot nds
\leq -\frac{d}{dt}\!\left(\int_{\partial\Omega}\!\! (u\cdot\nabla n\cdot u)Fds\right)\!+\frac{\delta}{6}\|\nabla\dot{u}\|_{L^2}^2\!+C\|\na u\|_{L^4}^4\!+CA_2^2\!+C.	
\end{split}\end{equation}
For the second term on the righthand side of \eqref{J10}, by \eqref{udot}, \eqref{f-om-h1}, \eqref{rho-inf1} and \eqref{low1}, we have
\begin{equation}\label{J12}
\begin{split}
& \int_{\partial\Omega}(u\cdot\nabla F)(\dot{u}\cdot n)ds =\int_{\partial\Omega}(u\cdot n^\bot)n^\bot\cdot\nabla F(\dot{u}\cdot n)  ds\\
=& \int_\Omega\nabla^\bot\cdot((u\cdot n^\bot) \nabla F(\dot{u}\cdot n) )dx
= \int_\Omega \nabla F\cdot\nabla^\bot((u\cdot n^\bot)(\dot{u}\cdot n) )dx\\
 \leq&C\int_\Omega |\nabla F||\dot{u}||\nabla u|dx+C\int_\Omega |\nabla F||\nabla \dot u||u|dx   \\
 \leq&C\|\nabla F\|_{L^2}\|\dot u\|_{L^4}\|\nabla u\|_{L^4} +C\|\nabla F\|_{L^2}\|\nabla\dot u\|_{L^2}\|\nabla u\|_{L^4}^{1/2} \\
\leq&\frac{\delta}{6} \|\nabla \dot u\|_{L^2}^2 +C\|\nabla u\|_{L^4}^2A_2^2+C,
\end{split}
\end{equation}
which implies that the last term of $J_1$ can be also bounded by the righthand side of \eqref{J12}.
Together with \eqref{J10}, \eqref{J11-3} and \eqref{J12}, we have
\begin{equation}\label{J1}
\begin{split}
& J_1 \leq -\mu\int (\mathrm{div}\dot{u})^{2}dx -\left(\int_{\partial\Omega} (u\cdot\nabla n\cdot u)Fds\right)_t+\frac{\delta}{3}\|\nabla\dot{u}\|\ltwo+\frac{\delta}{2}\|\nabla H_t\|\ltwo\\
& \quad +C  \|\nabla u\|^4_{L^4}+C (1+\|\nabla u\|^2_{L^4})A_2^2+C.
\end{split}
\end{equation}

Next, by $ \omega_t=\curl \dot u-u\cdot \na \omega-(\na^{\bot}u)^{T}:\nabla u$ and \eqref{navier-b}, a straightforward calculation leads to
\begin{equation}\label{J20}
\begin{split}
J_2 &=-\mu\int \curl\dot{u}\omega_tdx-\mu\int u \cdot \nabla \dot{u}\cdot(\nabla^{\bot}\omega)dx\\
&=-\mu\int |\curl \dot{u}|^{2}dx+\mu\int \curl\dot{u}(\na^{\bot}u)^{T}:\nabla u dx\\&\quad
+\mu\int \curl\dot{u}u\cdot \na \omega dx-\mu\int u \cdot \nabla \dot{u}\cdot(\nabla^{\bot}\omega)dx\\
&=-\mu\int |\curl \dot{u}|^{2}dx+\mu\int \curl\dot{u}(\na^{\bot}u)^{T}:\nabla u dx \\
&\quad+\mu\int \curl\dot{u}\,\mathrm{div}u \,\omega dx+\mu\int\nabla^{\bot} u \cdot \nabla \dot{u}\,\omega dx\\
&= -\mu\int |\curl \dot{u}|^{2}dx+\frac{\delta}{3} \|\nabla\dot{u}\|_{L^2}^2+C \|\nabla u\|_{L^4}^4.
\end{split}\end{equation}
Finally, by \eqref{g1} and \eqref{h-lp}, it shows that
\begin{equation}\label{J30}
\begin{split}
\displaystyle J_3
& =-\int \nabla\dot{u}:(H\otimes H)_t dx-\int H \cdot \nabla H_j u \cdot \nabla\dot{u}_{j}dx\\
& \leq C (\|\nabla\dot{u}\|_{L^2} \|H\|_{L^4} \|H_t\|_{L^4}+\|\nabla\dot{u}\|_{L^2} \|H\|_{L^8} \|\nabla H\|_{L^4}\|u\|_{L^8} ) \nonumber \\
& \leq \frac{\delta}{3} \|\nabla\dot{u}\|_{L^2}^2+\frac{\delta}{2}\|\nabla H_t\|\ltwo+C(\delta)(\|H_t\|_{L^2}^2+\|\Delta H\|_{L^2}^2).
\end{split}\end{equation}
Combining \eqref{J1}, \eqref{J20} with \eqref{J30}, we deduce from \eqref{J0} that
\begin{equation}\label{J01}
\begin{split}
&\frac{1}{2}\frac{d}{dt}\left(\|\sqrt{\rho}\dot{u}\|_{L^2}^2\right)+\mu \|\mathrm{div}\dot{u}\|_{L^2}^2+\mu \|\curl\dot{u}\|_{L^2}^2 \\
&\leq -\frac{d}{dt}\left(\int_{\partial\Omega} (u\cdot\nabla n\cdot u)Fds\right)+
\delta \|\nabla\dot{u}\|\ltwo+ \delta\|\nabla H_t\|\ltwo +C  \|\nabla u\|^4_{L^4}\\
& \quad +C (1+\|\nabla u\|_{L^4}^2)A_2^{2}+C\\
&\leq -\frac{d}{dt}\left(\int_{\partial\Omega} (u\cdot\nabla n\cdot u)Fds\right)+
\delta \|\nabla\dot{u}\|\ltwo+ \delta\|\nabla H_t\|\ltwo 
+C (1+A_2^2)A_2^{2}+C,
\end{split}\end{equation}
where in the last inequality we have used
\begin{equation}\label{tdu-l4}
\displaystyle \|\nabla u\|^4_{L^4}\leq C A_2^2+C,
\end{equation}
due to \eqref{tdu-lp} and \eqref{rho-inf}.

Next, we need to estimate the term $\|\nabla H_t\|_{L^2}$. Differentiating \eqref{CMHD}$_3$ with respect to $t$, multiplying it by $H_{t}$, and integrating over $\Omega$ lead to
\begin{align}\label{ht1}
&\quad \left(\frac{1}{2}\|H_t\|_{L^2}^2\right)_t+\nu \|\nabla H_t\|_{L^2}^2 \nonumber \\
 &= \int  (H_t \cdot \nabla u-u \cdot \nabla H_t-H_t \mathrm{div} u)\cdot H_t dx 
 + \int  (H \cdot \nabla \dot{u}-\dot{u} \cdot \nabla H-H \mathrm{div} \dot{u})\cdot H_t dx \nonumber \\
&\quad - \int  (H \cdot \nabla (u \cdot \nabla u)-(u \cdot \nabla u)\cdot \nabla H-H \mathrm{div}(u \cdot \nabla u) )\cdot H_t dx \nonumber \\
& \triangleq K_1+K_2+K_3.
\end{align}
By Lemma \ref{lem-gn} and Lemma \ref{lem-low1}, a direct calculation leads to
\begin{align}\label{htk1}
K_1 
& \leq C (\|H_t\|_{L^4}^2\|\nabla u\|_{L^2}
+\|u\|_{L^4}\|H_t\|_{L^4}\|\nabla H_t\|_{L^2})
\leq \frac{\delta}{3} \|\nabla H_t\|_{L^2}^{2}+C\|H_t\|_{L^2}^{2}.
\end{align}
Similarly, by \eqref{udot}, \eqref{h-lp} and \eqref{low1}, it follows
\begin{equation}\label{htk20}\begin{split}
K_2 
& \leq C \|H\|_{L^4}\|H_t\|_{L^4}\|\nabla \dot{u}\|_{L^2} +C \|\dot{u}\|_{L^4}\|H_t\|_{L^4}\|\nabla H\|_{L^2}\\
&\leq \delta\|\nabla \dot{u}\|_{L^2}^{2}+\frac{\delta}{3}\|\nabla H_t\|_{L^2}^{2}+C\|H_t\|_{L^2}^{2}+C.
\end{split}\end{equation}
By Sobolev trace theorem, \eqref{h-lp} and \eqref{low1}, we have
\begin{equation}\label{htk3}\begin{split}
K_3 
& =\int   H \cdot\nabla H_t \cdot( u \cdot \nabla u) dx
-\int   u \cdot \nabla u \cdot \nabla H_t \cdot H dx \\
& \leq C \|H\|_{L^8}\|\nabla H_t\|_{L^2}\|\nabla u\|_{L^4}\|u\|_{L^8} \\
&\leq \frac{\delta}{3} \|\nabla H_t\|_{L^2}^{2}+CA_2^2+C.
\end{split}\end{equation}
Putting \eqref{htk1}, \eqref{htk20} and \eqref{htk3} into \eqref{ht1}, we have
\begin{align}\label{ht3}
\frac{d}{dt}\left( \|H_t\|_{L^2}^2\right)+ \nu\|\nabla H_t\|_{L^2}^2\leq \delta(\|\nabla \dot{u}\|_{L^2}^{2}+\|\nabla H_t\|_{L^2}^{2})+CA_2^2+C.
\end{align}
Finally, by Lemma \ref{lem-gn} and \eqref{CMHD}$_3$, it holds
\begin{align}\label{h2xd1}\begin{split}
\|\Delta H\|_{L^2}&\leq C(\|H_t\|_{L^2}+\|\nabla H\|_{L^4}\|u\|_{L^4}+\|H\|_{L^4}\|\nabla u\|_{L^4}) \\
 &\leq C\|H_t\|_{L^2}+C\|\Delta H\|_{L^2}^{1/2}+CA_2^{1/2}+C\\
 &\leq \frac{1}{2}\|\Delta H\|_{L^2}+C\|H_t\|_{L^2}+C\|\sqrt{\rho}\dot{u}\|_{L^2}+C,
\end{split}
\end{align}
which implies
\begin{align}\label{h2xd2}\begin{split}
\|\Delta H\|_{L^2}\leq C(\|H_t\|_{L^2}+\|\sqrt{\rho}\dot{u}\|_{L^2})+C.
\end{split}
\end{align}
Choosing $\delta$ small enough, we deduce after adding \eqref{J01} and \eqref{ht3}  together that
\begin{align}\label{a20}\begin{split}
&\frac{d}{dt}\left(\|\sqrt{\rho}\dot{u}\|_{L^2}^2+\|H_t\|_{L^2}^2\right)+\|\nabla\dot{u}\|_{L^2}^2+\|\nabla H_t\|_{L^2}^2 \\
\leq &-\frac{d}{dt}\left(\int_{\partial\Omega} (u\cdot\nabla n\cdot u)Fds\right)+ C(1+A_2^2)A_2^2+C \\
\leq &-\frac{d}{dt}\left(\int_{\partial\Omega} (u\cdot\nabla n\cdot u)Fds\right)+ C(1+A_2^2)(\|\sqrt{\rho}\dot{u}\|_{L^2}^2+\|H_t\|_{L^2}^2)+CA_2^2+C.	
\end{split}\end{align}
Multiplying \eqref{a20} by $\sigma$, using \eqref{a1-i1}, \eqref{low1}, \eqref{h2xd2},  and Gronwall's inequality, we have
\begin{align}\label{a21}
&\quad \sup_{0\leq t\leq T}\sigma(\|\sqrt{\rho}\dot{u}\|_{L^2}^2+\|H_t\|_{L^2}^2)+\int_0^T\sigma(\|\nabla\dot{u}\|_{L^2}^2+\|\nabla H_t\|_{L^2}^2)dt\leq C.
\end{align}
Combining \eqref{h2xd2} and \eqref{a21}, we give \eqref{a02}.

Finally, it follows from \eqref{h-lp}, \eqref{tdu-lp},  and \eqref{rho-inf1} that for $p\geq 2$,
\begin{equation*}
\begin{split}
&\sigma\|\nabla u\|_{L^p}^2\leq C(p)\sigma\|\sqrt\rho\dot{u}\|^2_{L^2}+C(p)\leq C(p),\\
\end{split}
\end{equation*}
which shows \eqref{tdu-lp-t} and completes the proof of Lemma \ref{lem-a0}.
\end{proof}

\begin{lemma}\label{lem-low2}
Assume that \eqref{b-g} holds.  Then for any $p>2$,  there exists some positive constant $C$ depending only on $p$, $\Omega$, $T$, $\mu$, $\beta$, $\gamma$, $\nu$, $\|\rho_0\|_{L^\infty}$, $\|u_0\|_{H^1}$ and $\|H_0\|_{H^1}$ such that
\begin{equation}
\begin{split}\label{low2}
&\int_0^T\Big((\|\rho\dot{u}\|_{L^p}+\|\Delta H\|_{L^p})^{1+1/p}+t(\|\dot{u}\|_{H^1}^2+\|\Delta H\|_{L^p}^2)\Big)dt\leq C.
\end{split}
\end{equation}
Moreover,
\begin{equation}
\begin{split}\label{f-inf}
& \int_0^T\!\!\Big(\!(\|\nabla F\|_{L^p}\!+\!\|\nabla\omega\|_{L^p}\!+\!\|F\|_{L^\infty}\!+\!\|\omega\|_{L^\infty})^{1+\frac1p}\!+\!t(\|\nabla F\|_{L^p}^2\!+\!\|\nabla\omega\|_{L^p}^2)\Big)dt\leq C,
\end{split}
\end{equation}
and that
\be\label{inf-r}\inf_{(x,t)\in \Omega\times(0,T)}\rho(x,t)\ge C^{-1}\inf_{x\in \Omega}\rho_0(x).\ee
\end{lemma}
\begin{proof}
First, it follows from \eqref{udot}, \eqref{low1} and \eqref{a02} that
\begin{equation}\label{h1}
\int_0^Tt\|\dot{u}\|_{H^1}^2dt\leq C\int_0^Tt(\|\nabla\dot{u}\|_{L^2}^2+C)dt\leq C.
\end{equation}
By using \eqref{g1}, \eqref{udot}, \eqref{h-lp}, \eqref{rho-inf1} and \eqref{low1}, we have
\begin{equation*}
\begin{split}
\|\rho\dot{u}\|_{L^p} &\leq C\|\rho\dot{u}\|_{L^2}^{2(p-1)/(p^2-2)}\|\dot{u}\|_{L^{p^2}}^{p(p-2)/(p^2-2)}\\
&\leq C+ C\|\rho\dot{u}\|_{L^2} +C\|\rho\dot{u}\|_{L^2}^{2(p-1)/(p^2-2)}\|\nabla\dot{u}\|_{L^2}^{p(p-2)/(p^2-2)}, 
\end{split}
\end{equation*}
which together with \eqref{low1} and \eqref{a02} implies that
\begin{equation}\label{rho-du-lp}
\begin{split}
&\int_0^T\|\rho\dot{u}\|_{L^p}^{1+1/p}dt
\leq C\int_0^T\left(\|\sqrt\rho\dot{u}\|_{L^2}^2+t\|\nabla\dot{u}\|_{L^2}^2 +t^{-1+\frac{2}{p^3-p^2-2p+2}}\right)dt\leq C.
\end{split}
\end{equation}
From \eqref{CMHD}$_3$, \eqref{boundary}, \eqref{tdu-lp-t} and H\"{o}lder's inequality, we have
\begin{equation*}
\begin{split}
\|\Delta H\|_{L^p} &\leq C (\|H_t\|_{L^p}+\|H \cdot \nabla u\|_{L^p}+\|u \cdot \nabla H\|_{L^p}+\|H \mathop{\mathrm{div}}\nolimits u\|_{L^p})\\
&\leq C (\|H_t\|_{L^p}+\|H_t\|_{L^2}^{2/p}\|\nabla H_t\|_{L^{2}}^{(p-2)/p}
+\|\Delta H\|_{L^2}^{1/4}\sigma^{-1/2}+\|\Delta H\|_{L^2})+C,
\end{split}
\end{equation*}
which together with \eqref{low1} and \eqref{a02} implies that
\begin{equation}\label{td2h-lp}
\begin{split}
&\int_0^T\|\Delta H\|_{L^p}^{1+1/p}+t\|\Delta H\|_{L^p}^2dt\\
\leq& C\int_0^T\left(\|H_t\|_{L^2}^2+t\|\nabla H_t\|_{L^2}^2 +t^{-1+\frac{2}{p^2-p}}+\|\Delta H\|_{L^2}^2+\sigma^{-1+\frac{3p-5}{7p-1}}\right)dt+C\leq C.
\end{split}
\end{equation}
Then \eqref{h1}, \eqref{rho-du-lp} and \eqref{td2h-lp} gives \eqref{low2}.

Next, we obtain from \eqref{f-ow}, \eqref{low1}, \eqref{a02} and \eqref{low2} that
\begin{equation}\label{tdf-lp}
\begin{split}
& \int_0^T\!(\|\nabla F\|_{L^p}\!+\|\nabla\omega\|_{L^p}+t(\|\nabla F\|_{L^p}^2\!+\|\nabla\omega\|_{L^p}^2)dt\leq C,
\end{split}
\end{equation}
Then, \eqref{g2}, \eqref{f-ow}, \eqref{f-l2} and \eqref{low1} yield that
\begin{equation}\label{f-inf-2}
\begin{split}
\|F\|_{L^\infty} +  \|\omega\|_{L^\infty}&\leq C\|F\|_{L^2}+ C\|F\|_{L^2}^{\frac{p-2}{2(p-1)}}\|\nabla F\|^{\frac{p}{2(p-1)}}_{L^p}+ C\|\omega\|_{L^2}^{\frac{p-2}{2(p-1)}}\|\nabla \omega\|^{\frac{p}{2(p-1)}}_{L^p}  \\
&\leq C+C\|\rho\dot{u}\|^{\frac{p}{2(p-1)}}_{L^p}+C\|\Delta H\|^{\frac{p}{2(p-1)}}_{L^2} , \\
\end{split}
\end{equation}
which together with  \eqref{low2} and \eqref{f-inf-2} gives \eqref{f-inf}.
Similarly, from \eqref{flux}, \eqref{rho-inf1} and \eqref{low1}, it follows that
\begin{equation}\label{div-inf}
\begin{split}
\|\mathop{\mathrm{div}}\nolimits u\|_{L^\infty} &\leq C(\|F\|_{L^\infty} +  \|P\|_{L^\infty}+  \|H\|_{L^\infty}^2)\\
&\leq C+C\|\rho\dot{u}\|^{\frac{p}{2(p-1)}}_{L^p}+C\|\Delta H\|^{\frac{p}{2(p-1)}}_{L^2} , \\
\end{split}
\end{equation}
and \eqref{inf-r} is a direct consequence of  \eqref{div-inf}, \eqref{low2} and  \eqref{CMHD}$_1.$
\end{proof}

Upon now, we have finished the lower order a priori estimates, and will turn to the higher order ones. We follow \cite{fll2021} to derive our final a priori estimates with some modifications due to the magnetic field.
\begin{lemma}\label{lem-high}
Assume that \eqref{b-g} holds.  Then,  for $q > 2$, there exists a positive constant $\tilde C$ depending only on $q$, $\Omega$, $T$, $\mu$, $\beta$, $\gamma$, $\nu$, $\|\rho_0\|_{W^{1,q}}$, $\|u_0\|_{H^1}$ and $\|H_0\|_{H^1}$ such that
\begin{equation}\label{rho-h}
\sup_{0\leq t\leq T}\!\left(\|\rho\|_{W^{1, q}}\!\!+\!\|\rho_t\|_{L^2}\!\!+\!t\|u\|^2_{H^2}\right)\!+\!\int^T_0\!\!\!(\|\nabla^2 u\|_{L^q}^{1\!+\!\frac{1}{q}}\!\!+\!t\|\nabla^2 u\|_{L^q}^2\!\!+\!t\|u_t\|^2_{H^1})dt\! \leq \tilde C.  
\end{equation}
\end{lemma}
\begin{proof}
First,  denoting by $\Psi=(\Psi_1,  \Psi_2)$ with $\Psi_i\triangleq(2\mu + \lambda(\rho))\partial_i\rho$ $(i=1, 2)$,  one deduces from $\eqref{CMHD}_1$ that $\Psi_i$ satisfies
\begin{equation}\label{psi1}
\partial_t\Psi_i + (u\cdot\nabla)\Psi_i + (2\mu + \lambda(\rho))\nabla\rho\cdot \partial_iu + \rho \partial_iF + \rho\partial_iP+\frac{1}{2}\rho\partial_i|H|^2 + \Psi_i\mathrm{div}u = 0.
\end{equation}
For $q>2$,  multiplying \eqref{psi1} by $|\Psi|^{q-2}\Psi_i$ and integrating the resulting equation over $\Omega$,  we obtain after
integration by parts and using \eqref{navier-b} to cancel out boundary term that
\begin{equation}
\frac{d}{dt}\|\Psi\|_{L^q}\leq \tilde C(1+\|\nabla u\|_{L^\infty})\|\Psi\|_{L^q}+\tilde C\|\nabla F\|_{L^q}+\tilde C\|H \cdot\nabla H\|_{L^q}.  \label{psi2}
\end{equation}

Next,    we deduce from standard $L^p$-estimate for elliptic system with boundary condition \eqref{navier-b} and \eqref{f-ow} that
\begin{equation}\label{td2u2}
\ba
\|\nabla^2u\|_{L^q}\leq& \tilde C(\|\nabla \mathrm{div}u\|_{L^q}+\|\nabla\omega\|_{L^q})\\
\leq&\tilde C(\|\nabla((2\mu+\lambda)\mathrm{div}u)\|_{L^q}+\|\mathrm{div}u\|_{L^\infty}\|\nabla\rho\|_{L^q}+\|\nabla\omega\|_{L^q})\\
\leq&\tilde C(\|\mathrm{div}u\|_{L^\infty}+1)\|\nabla\rho\|_{L^q}+\tilde C\|\rho\dot{u}\|_{L^q}+\tilde C\|H \cdot\nabla H\|_{L^q}.
\ea
\end{equation}
Then it follows from Lemma  \ref{lem-bkm},  \eqref{f-inf-2}, \eqref{div-inf} and \eqref{td2u2} that
\begin{equation}
\begin{split}
\|\nabla u\|_{L^\infty}\leq& \tilde C(\|\mathrm{div}u\|_{L^\infty}+\|\omega\|_{L^\infty})\mathrm{ln}(e+\|\nabla^2u\|_{L^q})+\tilde C\|\nabla u\|_{L^2}+\tilde C\\
\leq& \tilde C(1+\|\rho\dot{u}\|_{L^q}+\|\Delta H\|_{L^2})\mathrm{ln}(e+\|\nabla\rho\|_{L^q}+\|\rho\dot{u} \|_{L^q}+\|\Delta H\|_{L^2})+\tilde C .  \label{td2u3}
\end{split}
\end{equation}
    Noticing that
\begin{equation}
2\mu\|\nabla\rho\|_{L^q}\leq\|\Psi\|_{L^q}\leq \tilde C\|\nabla\rho\|_{L^q},   \label{tdr-lp}
\end{equation} substituting \eqref{td2u3} into \eqref{psi2}, and using \eqref{f-ow}, one gets
\begin{equation*}
\frac{d}{dt}\mathrm{ln}(e+\|\Psi\|_{L^q})\leq \tilde C(1+\|\rho\dot{u}\|_{L^q}+\|\Delta H\|_{L^2})\mathrm{ln}(e+\|\Psi\|_{L^q})+\tilde C(\|\rho\dot{u}\|^{1+1/q}_{L^q}+\|\Delta H\|_{L^2}),
\end{equation*}
which together with Gronwall's inequality,  \eqref{low1},  \eqref{low2}, and \eqref{tdr-lp} yields that
\begin{equation}\label{tdr-lp2}
\sup_{0\leq t\leq T}\|\nabla\rho\|_{L^q}\leq \tilde C.
\end{equation}
Furthermore, \eqref{CMHD}$_1$ together with \eqref{tdr-lp2} and \eqref{low1} yields
\begin{equation}\label{tdr-lp3}
\|\rho_t\|_{L^q}\leq \tilde C(\|u \cdot \nabla \rho\|_{L^2}+\|\rho \mathop{\mathrm{div}}\nolimits u\|_{L^2})\leq \tilde C(\|\nabla\rho\|_{L^q}+1)\leq \tilde C.
\end{equation}
Combining \eqref{low1}, \eqref{low2}, \eqref{td2u2}, and \eqref{tdr-lp2}  gives
\begin{equation}
\int_0^T\left(\|\nabla^2u\|_{L^q}^{1+1/q}+t\|\nabla^2u\|_{L^q}^2\right)dt\leq \tilde C. \label{td2u-1}
\end{equation}

Moreover,  it follows from  \eqref{udot}, \eqref{tdudot}, \eqref{low1}, \eqref{a02}, \eqref{tdu-l4},  and \eqref{td2u-1} that
\begin{equation}\label{ut-h1}
\begin{split}
&\int_0^Tt\|u_t\|_{H^1}^2dt\leq\tilde C\int_0^Tt\left(\|\dot u\|^2_{H^1}+\|u\cdot \nabla u\|_{H^1}^2\right)dt\\ \leq& \tilde C\int_0^Tt\left(1+\|\nabla \dot u\|_{L^2}^2+\|\nabla u\|^4_{L^4}+\|u\|_{H^1}^2\|\nabla u\|_{L^q}^2\right)dt\leq\tilde C.
\end{split}
\end{equation}

Finally, we obtain from \eqref{low1},  \eqref{f-ow}, and \eqref{tdr-lp2} that
\begin{equation}\label{td2u-2}
\begin{split}
\|\nabla^2u\|_{L^2}\leq& \tilde C(\|\nabla \mathrm{div}u\|_{L^2}+\|\nabla\omega\|_{L^2})\\
\leq&\tilde C(\|\nabla(2\mu+\lambda)\mathrm{div}u)\|_{L^2}
+\|\mathrm{div}u\|_{L^{2q/(q-2)}}\|\nabla\rho\|_{L^q}+\|\nabla\omega\|_{L^2})\\
\leq&\frac{1}{2}\|\nabla^2 u\|_{L^2}+\tilde C\|\rho\dot{u}\|_{L^2}+\tilde C \|\Delta H\|_{L^2}+\tilde C,
\end{split}
\end{equation}
which together with \eqref{a02} gives
\begin{equation*}
\sup_{0\leq t\leq T}t\|\nabla^2u\|^2_{L^2}\leq \tilde C.
\end{equation*}
Combining this, \eqref{tdr-lp2}-\eqref{ut-h1}  and \eqref{low1} yields \eqref{rho-h} and finishes   the proof of Lemma \ref{lem-high}.
\end{proof}

\section{Proof of  Theorem  \ref{th1}-\ref{th2}}\label{se5}

With all the a priori estimates in Section \ref{se3} and Section \ref{se4} at hand, we are going to  prove the main result of the paper in this section.

We first state the global existence of strong solution $(\rho,  u, H)$ provided that \eqref{b-g} holds and that $(\rho_0,  m_0, H_0)$ satisfies \eqref{dt2} whose proof is similar to that of \cite[Proposition 5.1]{hl2016}.
\begin{proposition} \label{prop3}
Assume that \eqref{b-g} holds and that $(\rho_0,  m_0, H_0)$ satisfies \eqref{dt2}.  Then there exists a unique strong solution $(\rho,  u, H)$ to \eqref{CMHD}-\eqref{boundary} in $\Omega\times(0,  \infty)$ satisfying \eqref{esti-uh-local} for any $T\in(0, \infty)$.  In addition,  for any $q > 2$,  $(\rho,  u, H)$ satisfies \eqref{rho-h} with some positive constant $C$ depending only on $q$, $\Omega$, $T$, $\mu$, $\beta$, $\gamma$, $\nu$, $\|\rho_0\|_{W^{1,q}}$, $\|u_0\|_{H^1}$ and $\|H_0\|_{H^1}$.
\end{proposition}
\begin{proof}
By Lemma \ref{lem-local}, assume that the initial data $(\rho_0,m_0,H_0)$ satisfies \eqref{dt2}, there is a $T_0>0$ depending on $\inf\limits_{x\in \Omega}\rho_0(x)$ such that the problem \eqref{CMHD}-\eqref{boundary} has a unique local strong solution $(\rho,u,H)$ on $\Omega\times (0,T_0]$ satisfying \eqref{esti-uh-local} and \eqref{rho-local}. We set
\begin{equation}\label{sup-t}
	T^*=\sup \Big\{T\Big|\sup_{0\leq t\leq T}\|(\rho,u,H)\|_{H^2}<\infty\Big\}.
\end{equation}
Clearly, $T^*\geq T_0$.
We assert that
\begin{equation}
	T^*=\infty.
\end{equation}
Otherwise, $T^*<\infty$. Consequently, for any $T\in (0,T^*)$, \eqref{dt2} together with \eqref{inf-r} implies that
\begin{equation}\label{inf-r2}
\inf_{(x,t)\in \Omega\times(0,T)}\rho(x,t)\ge \hat{C}^{-1},
\end{equation}
where and in what follows, $\hat{C}$ denotes some generic positive constant depending on $T^*$ and $\inf\limits_{x\in \Omega}\rho_0(x)$ but independent of $T$.
Because of \eqref{dt2}, we define
\begin{equation}\label{in-data1}
	\begin{split}
		\sqrt{\rho}\dot{u}(x,t=0)=&\sqrt{\rho_0}(\mu \Delta u_0\!+\!\nabla((\mu\!+\!\lambda(\rho_0))\mathop{\mathrm{div}}\nolimits u_0\!-\!P(\rho_0)\!-\!|H_0|^2/2)\!+\!H_0 \!\cdot\! \nabla H_0),\\
H_t(x,t=0)=&\nu \Delta H_0+H_0 \cdot \nabla u_0-u_0 \cdot \nabla H_0-H_0 \mathop{\mathrm{div}}\nolimits u_0.
	\end{split}
\end{equation}
Integrating \eqref{a20} with respect to $t$ over $(0,T)$ together with \eqref{dt2}, \eqref{low1} and \eqref{in-data1} yields
\begin{equation}\label{a20-new}
\sup_{0\leq t\leq T}(\|\sqrt{\rho}\dot{u}\|_{L^2}^2+	\|H_t\|_{L^2}^2+\|\Delta H\|_{L^2}^2)+\int_0^T\|\nabla\dot{u}\|_{L^2}^2+	\|\nabla H_t\|_{L^2}^2)dt\leq \hat{C}.
\end{equation}
This combined with \eqref{f-ow}, \eqref{rho-h}, and \eqref{td2u-2} leads to
\begin{equation}\label{td2u-3}
	\begin{split}
&\sup_{0\leq t\leq T}(\|\nabla^2 u\|_{L^2}^2+	\|\nabla F\|_{L^2}^2+\|\nabla \omega\|_{L^2}^2)+\int_0^T(\|\nabla^2F\|_{L^2}^2+	\|\nabla^2 \omega\|_{L^2}^2)dt\\
\leq& \hat{C}\sup_{0\leq t\leq T}(\|\sqrt{\rho}\dot{u}\|_{L^2}^2+\|\Delta H\|_{L^2}^2)+\hat{C}\int_0^T(\|\nabla(\rho\dot{u})\|_{L^2}^2+	\|\nabla^2|H|^2\|_{L^2}^2)dt\\
\leq&\hat{C}+\hat{C}\int_0^T(\|\nabla \rho\|_{L^q}^2\|\dot{u}\|_{L^{\frac{2q}{q-2}}}^2+\|\nabla\dot{u}\|_{L^2}^2)dt \leq \hat{C},
	\end{split}
\end{equation}
where we have used the fact
\begin{equation}\label{td2h-new}
\|\nabla^2|H|^2\|_{L^2}^2\leq \hat{C} (\|\nabla H\|_{L^4}^4+\|\Delta H \cdot H\|_{L^2}^2)\leq \hat{C}+\hat{C}\|\Delta H\|_{L^2}^{5/2}\leq \hat{C}.
\end{equation}

Next, we claim that, for all $0 <T<T^*$,
 \begin{equation}\label{sup-r-h2}
 	\sup_{0\leq t\leq T}\|\rho\|_{H^2}<\hat{C},
 \end{equation}
which together with \eqref{a20} and \eqref{td2u-3} contradicts \eqref{sup-t}. To finish the proof of Proposition \ref{prop3}, it only remains to prove \eqref{sup-r-h2}. This is done in essentially the same plan as used in Lemma \ref{lem-high}.
Operating $\nabla$ to \eqref{psi1} and multiplying the resulting equality by $\nabla\Psi_i$, we obtain after integration by parts and using \eqref{inf-r2} and \eqref{td2u-3} that
\begin{equation}\label{psi2-1}
\begin{split}
\frac{d}{dt}\|\nabla\Psi\|_{L^2}\leq& \hat{C}(1+\|\nabla u\|_{L^\infty})(1+\|\nabla\Psi\|_{L^2}+\|\nabla \rho\|_{L^4}^2+\|\nabla^2 \rho\|_{L^2})\\
&+\hat{C}\||\nabla \rho||\nabla^2 u|\|_{L^2}+\hat{C}\||\nabla \rho||\nabla F|\|_{L^2}+\hat{C}\|\nabla^2 F\|_{L^2}\\
&+\hat{C}\||\nabla \rho||\nabla|H|^2|\|_{L^2}+\hat{C}\|\nabla^2 |H|^2\|_{L^2}\\
\leq& \hat{C}(1+\|\nabla u\|_{L^\infty})(1+\|\nabla\Psi\|_{L^2})+\hat{C}\|\nabla^2 F\|_{L^2}+\hat{C}\|\nabla^2 \omega\|_{L^2},	
\end{split}
\end{equation}
due to standard $L^2$-estimate for elliptic system, \eqref{td2u-3}, \eqref{td2h-new} and the fact
\begin{equation}\label{td2r}
\begin{split}
\|\nabla \rho\|_{L^4}^2+\|\nabla^2 \rho\|_{L^2}
\leq\hat{C} \|\nabla \Psi\|_{L^2}+\hat{C}\|\nabla \rho\|_{L^4}^2
\leq\hat{C} \|\nabla \Psi\|_{L^2}+\frac{1}{2}\|\nabla^2 \rho\|_{L^2}+\hat{C}.
\end{split}
\end{equation}
Then, \eqref{psi2-1} combined with \eqref{td2u-3} and Gronwall's inequality yields
\begin{equation}\label{sup-phi2}
\sup_{0\leq t\leq T}\|\nabla \Phi\|_{L^2}<\hat{C},	
\end{equation}
which together with \eqref{td2r} implies \eqref{sup-r-h2}. The proof of Proposition \ref{prop3} is finished.
\end{proof}

{\it\textbf{Proof of Theorem \ref{th1}.}}
Let $(\rho_0,  m_0, H_0)$ satisfying \eqref{dt1} be the initial data as described in Theorem \ref{th1}.  We construct a sequence of $C^{\infty}$ initial value $(\rho_0^{\delta}, u_0^{\delta}, H_0^{\delta})$,  and $(u_0^{\delta}, H_0^{\delta})$ should satisfy boundary condition.   Using the standard approximation theory to $({\rho}_0, u_0, H_0)$(see \cite{evans2010} for example), we can find
a sequence of $C^\infty$ functions $(\tilde{\rho}_0^\delta,\tilde{u}_0^\delta,\tilde{H}_0^\delta)$ satisfying
\begin{equation}\label{51}
\lim_{\delta\rightarrow 0}\|\tilde{\rho}_0^\delta-\rho_0\|_{W^{1,q}}+\|\tilde{u}_0^\delta-u_0\|_{H^1}+\|\tilde{H}_0^\delta-H_0\|_{H^1}=0.
\end{equation}
Furthermore, let $(u^\delta_0,H^\delta_0)$ be the unique smooth solution of elliptic system
\begin{equation}\label{52}
\begin{cases}
  \Delta u^{\delta}=\Delta\tilde{u}^\delta_0,\quad \Delta H^{\delta}=\Delta\tilde{H}^\delta_0, & \mbox{ in } \Omega, \\
   u^{\delta}\cdot n=0, \ \mathrm{curl}u^{\delta}=0, \ H^{\delta}=0,& \mbox{ on }\partial\Omega. \\
       \end{cases}
\end{equation} We define $ {\rho}_0^\delta=\tilde\rho^\delta_0+\delta$,
and set $m^\delta_0=\rho^\delta_0 u^\delta_0.$   Then, it is easy to check that
\begin{equation*}
\lim_{\delta\rightarrow 0}(\|\rho^\delta_0-\rho_0\|_{W^{1, q}}+\|u^\delta_0-u_0\|_{H^1}+\|H^\delta_0-H_0\|_{H^1})=0.
\end{equation*}

After applying Proposition \ref{prop3}, we can construct a unique global strong solution $(\rho^{\delta}, u^{\delta}, H^{\delta})$ with initial value $(\rho_0^{\delta}, u_0^{\delta}, H_0^{\delta})$.
Such solution satisfy \eqref{rho-h} for any $T>0$ and for some $C$ independent with $\delta$.
Letting $\delta\rightarrow 0$, standard compactness assertions (see\cite{lzz2015,vk1995,Pere2006}) make sure the problem \eqref{CMHD}-\eqref{boundary} has a global strong solution $(\rho, u, H)$ satisfying \eqref{esti-uh}.
Since the uniqueness can be obtained via similar method in Germain \cite{germ2011}, we omit the details and finish the proof of Theorem \ref{th1}.   \endproof

{\it \textbf{Proof of Theorem \ref{th2}.}}
With all a priori estimations and Theorem \ref{th1} done, we may follow the proof (especially the compactness assertions) in \cite{hl2016} to deduce
the existence of weak solution. We just sketch the proof for completeness.
Let $(\rho_0, m_0, H_0)$ be the initial data as in Theorem \ref{th2}, we construct an approximation initial value
$(\rho^\delta_0, u^\delta_0, H^\delta_0)$ in the same manner as \eqref{51} and \eqref{52}, but this time we have for any $p\geq 1$,
\begin{equation*}
\lim_{\delta\rightarrow 0}(\|\rho^\delta_0-\rho_0\|_{L^{p}}+\|u^\delta_0-u_0\|_{H^1}+\|H^\delta_0-H_0\|_{H^1})=0.
\end{equation*}
Moreover,
\begin{equation*}
\rho^\delta_0\rightarrow\rho_0\  \mathrm{in}\ W^*\ \mathrm{topology}\ \mathrm{of}\ L^\infty\ \mathrm{as}\ \delta\rightarrow 0.
\end{equation*}
We apply Proposition \ref{prop3} to deduce the existence of unique global strong solution $(\rho^\delta, u^\delta, H^\delta)$ of problem \eqref{CMHD}-\eqref{boundary} with initial value $(\rho^\delta_0, u^\delta_0, H^\delta_0)$ which satisfy \eqref{rho-inf1}, \eqref{low1}, \eqref{a02}, \eqref{tdu-lp-t}, \eqref{low2}
and \eqref{f-inf} for any $T>0$. The constants involved are all independent of $\delta$. Combining \eqref{low1} and \eqref{low2} yields that
\begin{equation*}
\sup_{0\leq t\leq T}\|(u^\delta, H^\delta)\|_{H^1}+\int_0^T t\|(u^\delta_t, H^\delta_t)\|_{L^2}^2dt\leq C.
\end{equation*}
Applying Aubin-Lions Lemma, for any $\tau\in (0, T)$ and $p\geq 1,$ one gets that, up to a subsequence,
\begin{equation*}
\begin{split}
&(u^\delta, H^\delta)\rightharpoonup (u, H), \mbox{ weakly  * } \mathrm{in}\ L^\infty(0, T;H^1), \\
&(u^\delta, H^\delta)\rightarrow (u, H)\  \mathrm{in}\ C([\tau, T]; L^p).
\end{split}
\end{equation*}

Moreover, denote $F^\delta=\big(2\mu+(\rho^\delta)^\beta\big)\mathrm{div} u^\delta-P(\rho^\delta)-\frac{1}{2}|H^\delta|^2$ be the effect viscous flux with respect to approximation solution $(\rho^\delta, u^\delta, H^\delta)$.
From \eqref{low1} and \eqref{f-inf}, we have
\begin{equation*}
\int_0^T\left(\big \|F^\delta\|_{L^\infty}^{ {4}/{3}} +\|\omega^\delta\|_{H^1}^2 +\|F^\delta\|_{H^1}^{2}
+t\|\omega^\delta_t\|_{L^2}^2+t\|F^\delta_t\|_{L^2}^{2}\right) dt\leq C.
\end{equation*}
Similarly, Aubin-Lions Lemma yields
\begin{equation}
\begin{split}\label{59}
&F^\delta\rightharpoonup F \mbox{ weakly  * }  \ \mathrm{ in }\ L^{\frac{4}{3}}(0, T;L^{\infty}),\\
&(F^\delta, \omega^\delta)\rightarrow (F, \omega), \ \mathrm{in}\ L^{2}(\tau, T;L^{p}),
\end{split}
\end{equation}
for any  $\tau\in (0, T)$  and $p\geq 1. $

Now we may follow the assertions in \cite{hl2016} to deduce the strong convergence of $\rho^\delta$, say
\begin{equation*}
\rho^\delta\rightarrow \rho\ \mathrm{ in }\ C([0,  T];L^p(\Omega)),
\end{equation*}
for any $p\geq 1$. Combining this with \eqref{59}, we argue that
\begin{equation*}
F^\delta\rightarrow F\ \mathrm{in}\ L^2(\Omega\times(\tau,  T)),
\end{equation*}
for any $\tau\in (0, T).$  And \eqref{esti-w} follows directly.
We conclude that $(\rho, u, H)$ is exactly the desired weak solution in Theorem \ref{th2} and the proof is finished.   \endproof
\section*{Acknowledgements} This research was partially supported by National Natural Sciences Foundation of China No. 11901025, 11671027, 11971020, 11971217.

\end{document}